\documentclass[1 [leqno,11pt]{amsart}
\usepackage{amssymb, amsmath}
\def\refer#1{~\ref{#1}}
\def\refeq#1{~(\ref{#1})}
\def\ccite#1{~\cite{#1}}

\def\longformule#1#2{
\displaylines{ \qquad{#1} \hfill\cr \hfill {#2} \qquad\cr } }
\def\inte#1{
\displaystyle\mathop{#1\kern0pt}^\circ }

\def\vect#1{
\overrightarrow{#1} }

\let\pa=\partial
\let\al=\alpha

\let\d=\delta

\let\lam=\lambda

\let\s=\sigma
\let\f=\frac

\let\p=\psi
\let\om=\omega

\let\D=\Delta

\let\Om=\Omega
\let\wt=\widetilde
\let\wh=\widehat

\def\cA{{\mathcal A}}
\def\cB{{\mathcal B}}
\def\cC{{\mathcal C}}
\def\cD{{\mathcal D}}
\def\cE{{\mathcal E}}
\def\cF{{\mathcal F}}

\def\cH{{\mathcal H}}

\def\cS{{\mathcal S}}

\def\grad{\nabla}

\def\dH{{H}}

\def\virgp{\raise 2pt\hbox{,}}
\def\cdotpv{\raise 2pt\hbox{;}}

\def\eqdefa{\buildrel\hbox{\footnotesize def}\over =}
\def\Id{\mathop{\rm Id}\nolimits}

\def\C{\mathop{\mathbb C\kern 0pt}\nolimits}
\def\DD{\mathop{\mathbb D\kern 0pt}\nolimits}
\def\EE{\mathop{{\mathbb E \kern 0pt}}\nolimits}
\def\K{\mathop{\mathbb K\kern 0pt}\nolimits}
\def\N{\mathop{\mathbb N\kern 0pt}\nolimits}
\def\Q{\mathop{\mathbb Q\kern 0pt}\nolimits}
\def\R{\mathop{\mathbb R\kern 0pt}\nolimits}
\def\SS{\mathop{\mathbb S\kern 0pt}\nolimits}
\def\ZZ{\mathop{\mathbb Z\kern 0pt}\nolimits}
\def\TT{\mathop{\mathbb T\kern 0pt}\nolimits}
\def\P{\mathop{\mathbb P\kern 0pt}\nolimits}
\newcommand{\ds}{\displaystyle}

\newcommand{\Z}{{\ZZ}}

\def\dv{\mbox{\rm div}}

\def\dive{\mathop{\rm div}\nolimits}

\def\Supp{\mathop{\rm Supp}\nolimits\ }

\def\no{\noindent}
\def\na{\nabla}
\def\p{\partial}
\def\th{\theta}

\def\vcurl{v^{\rm h}_{\rm curl}}
\def\nablah{\nabla_{\rm h}}
\def\vdiv{v^{\rm h}_{\rm div}}

\newcommand{\beq}{\begin{equation}}
\newcommand{\eeq}{\end{equation}}
\newcommand{\ben}{\begin{eqnarray}}
\newcommand{\een}{\end{eqnarray}}
\newcommand{\beno}{\begin{eqnarray*}}
\newcommand{\eeno}{\end{eqnarray*}}
\newcommand{\andf}{\quad\hbox{and}\quad}
\newcommand{\with}{\quad\hbox{with}\quad}
\newtheorem{defi}{Definition}[section]
\newtheorem{thm}{Theorem}[section]
\newtheorem{lem}{Lemma}[section]

\newtheorem{prop}{Proposition}[section]

\begin{document}
\title[Regularity criterion for 3-D
 Navier-Stokes system]
{On the critical one component regularity for 3-D
 Navier-Stokes system}
 \author[J.-Y. CHEMIN]{Jean-Yves Chemin}
\address [J.-Y. Chemin]%
{Laboratoire J.-L. Lions, UMR 7598 \\
Universit\'e Pierre et Marie Curie, 75230 Paris Cedex 05, FRANCE }
\email{chemin@ann.jussieu.fr}
\author[P. ZHANG]{Ping Zhang}%
\address[P. Zhang]
 {Academy of
Mathematics $\&$ Systems Science and  Hua Loo-Keng Key Laboratory of
Mathematics, The Chinese Academy of Sciences\\
Beijing 100190, CHINA } \email{zp@amss.ac.cn}
\date{2/06/2013}
\maketitle

\begin{abstract} Given an initial data $v_0$
with vorticity~$\Om_0=\na\times v_0$ in~$L^{\frac 3 2},$ (which
implies that~$v_0$ belongs to the Sobolev space~$H^{\frac12}$),  we
prove that the solution~$v$ given by the classical Fujita-Kato
theorem  blows up in a finite time~$T^\star$   only if, for any $p$
in~$ ]4,6[$ and any unit vector~$e$ in~$\R^3,$ there holds $
\int_0^{T^\star}\|v(t)\cdot e\|_{\dH^{\f12+\f2p}}^p\,dt=\infty.$ We
remark that all these quantities  are scaling invariant under the
scaling transformation of Navier-Stokes system.
\end{abstract}

Littlewood-Paley Theory\

\vskip 0.2cm


\setcounter{equation}{0}
\section{Introduction}

In the present work, we investigate necessary conditions for the
breakdown of  the regularity of regular solutions to the following
3-D homogeneous incompressible Navier-Stokes system
\begin{equation*}
(NS)\qquad \left\{\begin{array}{l}
\displaystyle \pa_t v + \dv (v\otimes v) -\D v+\grad \Pi=0, \qquad (t,x)\in\R^+\times\R^3, \\
\displaystyle \dv\, v = 0, \\
\displaystyle  v|_{t=0}=v_0,
\end{array}\right. \label{1.1}
\end{equation*}
where $v=(v^1,v^2, v^3)$ stands for the   velocity of the fluid and
$\Pi$ for the pressure.  Let us first recall some fundamental
results proved by J. Leray in his seminal  paper\ccite{Leray}.
\begin{thm}
\label{leraysemiregul}
{\sl Let us consider an initial data~$v_0$
 which belongs to the  inhomogeneous Sobolev space~$H_{\rm in}^1(\R^3)$. There exists a (unique) maximal positive time of existence~$T^\star$
 such that a unique solution~$v$ of $(NS)$ exists on~$[0,T^\star[\times\R^3,$ which is continuous with value in~$H_{\rm in}^1(\R^3)$
 and the gradient of which belongs to~$L^2_{\rm loc}([0,T^\star[;H^1_{\rm in}(\R^3))$. Moreover, if~$\|v_0\|_{L^2}\|\nabla v_0\|_{L^2}$ is
  small enough, then~$T^\star$ is infinite. If~$T^\star$ is finite, we have, for any~$q$ greater than~$3$,
$$
\forall\ t<T^\star\,,\ \|v(t)\|_{L^q} \geq  \frac {C_q}
{(T^\star-t)^{\frac 12\left(1-\frac3q\right)}}\,\cdotp
$$
}
\end{thm}

Let us also mention that  in\ccite{Leray}, J. Leray proved also the
existence  (but not the uniqueness) of global weak (turbulent in J.
Leray's terminology) solutions of $(NS)$ with initial data only in
$L^2(\R^3)$. In the present paper, we only deal  with solutions
which are regular to be unique.

In\ccite{Leray}, J. Leray  emphasized two basic facts about the
homogeneous incompressible Navier-Stokes system: the $L^2$ energy
estimate and the scaling invariance.

\medbreak

Because the vector field~$v$ is divergence free, the energy estimate
formally reads
$$
\frac 12\frac d {dt} \|v(t)\|_{L^2}^2  +\|\nabla v(t)\|_{L^2}^2 =0.
$$
After time  integration, this  gives \beq \label{energynormeq} \frac
12 \|v(t)\|_{L^2}^2  +\int_0^t\|\nabla v(t')\|_{L^2}^2dt' = \frac 12
\|v_0\|_{L^2}^2. \eeq This estimate is the cornerstone of the proof
of the existence of global turbulent solution to $(NS)$ done by J.
Leray in\ccite{Leray}.  The energy estimate  relies (formally) on
the fact that if~$v$ is a divergence free vector field,~$ (v\cdot
\nabla f|f)_{L^2} =0~$ and that~$ (\nabla p|v)_{L^2}=0$. In the
present work, we shall use the more general fact that for any
divergence free vector field~$v$ and any function~$a$, we have
$$
\int_{\R^3} v(x) \cdot \nabla a(x) |a(x)|^{p-2}a(x)\, dx =
0\qquad\mbox{for any}\ \ p\in ]1,\infty[.
$$
This will lead to the ~$L^p$ type energy  estimate.

\medbreak The scaling invariance is the fact that if~$v$ is a
solution of $(NS)$ on~$[0,T]\times \R^3$ associated with an initial
data~$v_0$, then~$\lam v(\lam^2t,\lam x)$ is also a solution of
$(NS)$ on~$[0,\lam^{-2}T]\times \R^3$ associated with the initial
data~$\lam v_0(\lam x)$ . The importance of this point can be
illustrated by this sentence coming from\ccite{Leray} {\it`` \dots
les \'equations aux dimensions permettent de pr\'evoir a priori
presque toutes les in\'egalit\'es que nous \'ecrirons \dots
"}\footnote{This can be translated by {\it "The scaling allows to
guess almost all the inequalities written in this paper"}} The
scaling property is also the foundation of the Kato theory  which
gives a general method to solve (locally or globally) the
incompressible Navier-Stokes equation  in critical spaces i.e.
spaces with the norms of which are invariant under the scaling. In
the present work, we only use  such scaling invariant spaces. Let us
exhibit some examples of scaling invariant norms. For~$p\geq 2$, the
norms of
$$
 L^p_t(H^{\frac 12 +\frac 2 p})\andf L^p_t ( L_x^{3+\frac {6} {p-2}}).
$$
are scaling invariant norms. The spaces~$H^{\frac 12}$ are~ $ L^3$
are scaling invariant spaces for the initial data~$v_0$ . Let us
point out that in the case when the space dimension is two,  the
energy norm  which appears in Relation\refeq{energynormeq} is
scaling invariant. This allows to prove that in the two dimensional
case, turbulent solutions are unique and regular.

The first result  of local (and global for small initial data)
wellposedness of $(NS)$ in a scaling invariant space was proved by
H. Fujita and T. Kato in~1964 (see\ccite{fujitakato}) for initial
data in the homogenenous Sobolev space~$H^{\frac12}$.  More
precisely, we have the following statement.
\begin{thm}
\label{fujitakatotheo} {\sl Let us consider an initial data~$v_0$ in
the homogeneous Sobolev space~$H^{\frac 12}(\R^3)$. There exists a
(unique) maximal positive time of existence~$T^\star$ such that a
unique solution~$v$  of $(NS)$ exists on~$[0,T^\star[\times\R^3$
which is continuous in time  with value in~$H^{\frac 12} (\R^3)$ and
belongs to~$L^2_{\rm loc}([0,T^\star[;H^\frac 32(\R^3))$. Moreover,
if the quantity~$\|v_0\|_{H^{\frac12}}$ is small enough,
then~$T^\star$ is infinite. If~$T^\star$ is finite, we have, for
any~$q$ greater than~$3$,
$$
\forall\ t<T^\star\,,\ \|v(t)\|_{L^q} \geq C_q \frac 1
{(T^\star-t)^{\frac 12\left(1-\frac3q\right)}}\,\cdotp
$$
}
\end{thm}

Let us point out  that the above necessary condition for blow up
implies that \beq \label{blowupFond} T^\star<\infty\Longrightarrow
\int_0^{T^\star} \|v(t)\|_{L^q} ^pdt =\infty\with \frac 2 p+\frac 3
q=1\andf p<\infty. \eeq Let us mention that it is possible to prove
this theorem without using the energy estimate and this theorem is
true for a large class of systems which have the same scaling as the
incompressible Navier-Stokes system.

Using results related to the energy estimate, L. Iskauriaza,  G. A.
Ser\"egin and V. Sverak proved in 2003 the end point case
of\refeq{blowupFond}  when~$p$ is infinite (see\ccite {ISS}). This
remarkable result has been extended to Besov space with negative
index (see\ccite{cheminplanchon}). Let us also mention a blow up
criteria  proposed by Beir$\tilde{\rm a}$o da Veiga \cite{Beirao},
which states that if the maximal time ~$T^\star$ of existence of a
regular solution $v$ to $(NS)$  is finite, then we have \beq
\label{blowupBeiraVega}
 \int_0^{T^\star} \|\na v(t)\|_{L^q} ^p dt =\infty \with
\f2p+\f3q= 2\quad\mbox{for}\quad q\geq\f32\,\cdotp \eeq Let us
observe that because of the fact that homogeneous bounded Fourier
multipliers  maps~$L^p$ into~$L^p$, this criteria is equivalent,
for~$q$ is finite, to \beq \label{blowupBeiraVegabis}
 \int_0^{T^\star} \|\Om(t)\|_{L^q} ^p dt =\infty \quad\hbox{where}\quad \Om\eqdefa \nabla\times v.
\eeq In this case when~$q$ is infinite, this criteria is the
classical Beale-Kato-Majda theorem (see \cite{BKM}) which is in fact
a result about Euler equation and where the viscosity plays no role.

\medbreak In the present paper, we want to establish necessary
conditions for breakdown of  regularity of  solutions to $(NS)$
given by Theorem\refer{fujitakatotheo} in term of the scaling
invariant norms of one component of the velocity field. Because we
shall use  the $L^{\frac 32}$ norm of the vorticity, we work with
solution given by the following theorem, which are a little bit more
regular than that given by Theorem \ref{fujitakatotheo}.
\begin{thm}
\label{existenceomegaL3/2}
{\sl Let us consider an initial data
$v_0$ with vorticity~$\Om_0=\na\times v_0$ in~$L^{\frac 3 2}$. Then
 a unique maximal solution~$v$ of~$(NS)$ exists in the
 space~$ C([0,T^\ast[;H ^{\f12})\cap
L^2_{\rm loc}([0,T^\star[;H ^{\frac32})$ for some positive time
$T^\ast,$ and the vorticity ~$\Om=\na\times v$ is continuous on
$[0,T^\ast[$ with value in~$L^{\frac 32 }$ and~$\Om$ satisfies
$$
|\nabla \Om|\, |\Om|^{-\frac14} \in L^2_{\rm loc}([0,T^\star[; L^2).
$$
}
\end{thm}

This theorem is classical. For the reader's convenience, we prove it
in the third section where we insist on the importance
of~$L^{\frac32}$ energy estimate for the vorticity.

\medbreak
The main theorem of this paper is the following.
\begin{thm}
\label{theoxplosaniscaling} {\sl We consider a maximal solution~$v$  of
$(NS)$ given by Theorem \ref{existenceomegaL3/2}. Let ~$p$ be
in~$]4,6[,$ ~$e$ a unit vector of~$\R^3,$ and
 $v_e\eqdefa v\cdot e$. Then  if  $T^\star<\infty,$ we have
\beq \label{k.1}
\int_0^{T^\star}\|v_e(t)\|_{\dH^{\f12+\f2p}}^p\,dt=\infty. \eeq }
\end{thm}

 Let us remark  that the quantity~$\ds \int_0^{T}\|v_e(t)\|_{\dH^{\f12+\f2p}}^p\,dt$ is
scaling invariant. Moreover, it  gives a necessary blow up condition
which involves only a scaling invariant norm to  one component of
the velocity. Or equivalently, it claims that if the maximal time of
existence~$T^\star$ is finite, $v$ blows up in any direction and thus is in some sense isotropic.

\medbreak Let us mention that  I. Kukavica and M.  Ziane proved
in\ccite{KZ2} further that
$$
T^\star<\infty\Longrightarrow  \int_0^{T^\star}
\|\p_3v(t,\cdot)\|^p_{L^q}dt=\infty \with \f2p+\frac 3 q= 2\andf
q\in [9/4,3]
$$
We notice that all these criteria concern scaling invariant norm.

\medbreak Recently, a lot of works (see \cite{CT1, CT,H,  KZ, NP,
NNP, PP, P, SK, ZP}) established constitution of the type
$$
 \int_0^{T^\star} \|v^3(t,\cdot)\|^p_{L^q}dt=\infty \quad\hbox{or}\quad
 \int_0^{T^\star} \|\partial_j v^3(t,\cdot)\|^p_{L^q}dt=\infty
 $$
 with conditions on~$p$ and~$q$ which make these quantities not scaling invariant.

\setcounter{equation}{0}
\section{Ideas and structure of the proof}
First of all, let us  remark that it makes no restriction to assume
that the unit vector~$e$ is the vertical vector~$(0,0,1)$. The first
idea of the present work  consists in  writing the incompressible
homogeneous Navier-Stokes  system in terms of two unknowns:
\begin{itemize}
\item
the third component of the vorticity~$\Om$, which we denote by
$$
\om = \partial_1v^2-\partial_2v^1
$$
and which can be understood as the 2D vorticity for the vector field~$v^{\rm h} \eqdefa (v^1,v^2)$,

\item the quantity~$\partial_3v^3$ which is~$-\dive_hv^{\rm h}=- \partial_1v^1-\partial_2v^2$ because~$v$ is divergence free.
\end{itemize}
 Immediate computations gives
$$
(\wt {NS})
\quad\left\{
\begin{array}{c}
\partial_t\om+v\cdot\nabla\om -\D\om= \partial_3v^3\om +\partial_2v^3\partial_3v^1-\partial_1v^3\partial_3v^2\\
\partial_t \partial_3v^3 +v\cdot\nabla\partial_3v^3-\D\partial_3v^3+\partial_3v\cdot\nabla v^3
=-\partial_3^2\D^{-1} \Bigl(\ds\sum_{\ell,m=1}^3 \partial_\ell
v^m\partial_mv^\ell\Bigr).
\end{array}
\right.
$$
Keeping in mind that we control~$v^3$ in the
norm~$L^p_T\bigl(H^{\frac12+ \frac 2p}\bigr)$ with $p$ greater
than~$4$, which implies that the order of regularity in space
variables is less than~$1$. Let us analyze this system. We first
introduce the notations \beq \label{a.1} \nabla_{\rm
h}^\perp=(-\p_2,\p_1),\quad\D_{\rm h}=\p_1^2+\p_2^2,\quad \vcurl
\eqdefa\nablah^\perp \D_{\rm h}^{-1} \om \andf \vdiv\eqdefa
-\nablah\D_{\rm h}^{-1}
\partial_3v^3.
 \eeq
 Then we have, using the Biot-Savart's law in the horizontal variables
\beq
\label{a.1wrt} v^{\rm h}=\vcurl+\vdiv.
 \eeq
Thus the righthand side term of the equation on~$\om$ in $(\wt{NS})$
contains terms which are linear in~$\om$, namely
$$
\partial_3v^3\om +\partial_2v^3\partial_3v_{\rm curl}^1-\partial_1v^3\partial_3v_{\rm
curl}^2,
$$
and a term that appears as a forcing term, namely
$$
\partial_2v^3\partial_3v_{\rm div}^1-\partial_1v^3\partial_3v_{\rm
div}^2.
$$
The only quadratic term  in~$\om$ is~$\vcurl\cdot\nablah\om$. A way to get rid of it is to
 use an energy type estimate and the divergence free condition on~$v$.
 As we want to work only with scaling invariant norms, the only way is to perform a
  $L^{\frac32}$ energy estimate in the equation on~$\om$.   Then it seems reasonable
  to control ~$\om$ using some norm on~$v^3$. Unfortunately,
     as shown by the forthcoming Proposition\refer{inegfondvroticity2D3D},
   we need higher order regularity  on~$v^3$. This leads to investigate the second equation of~$(\wt {NS}),$ which is
$$
 \partial_t \partial_3v^3 +v\cdot\nabla\partial_3v^3-\D\partial_3v^3+\partial_3v\cdot\nabla v^3
=-\partial_3^2\D^{-1} \Bigl(\ds\sum_{\ell,m=1}^3 \partial_\ell
v^m\partial_mv^\ell\Bigr).
$$
The main feature of this equation is that it contains only one
quadratic term with respect to~$\om$, namely the term
$$
-\partial_3^2\D^{-1} \Bigl(\ds\sum_{\ell,m=1}^2 \partial_\ell
v_{\rm curl}^m\partial_mv_{\rm curl}^\ell\Bigr)
$$
 A way to get rid of this term  is to perform an energy estimate on ~$\partial_3v^3$, namely an estimate on
 $$
\|\partial_3 v^3(t)\|_{\cH}
 $$
 for an adapted Hilbert space $\cH.$  Indeed, we  hope that if we control~$v^3,$ we can  control
terms of the type
$$
\bigl(\partial_3^2\D^{-1} (\partial_\ell
v_{\rm curl}^m\partial_mv_{\rm curl}^\ell)\big| \partial_3v^3\bigr)_{\cH}
$$
with quadratic terms in~$\om$ and thus it fits
with~$\|\partial_3v^3\|_{\cH}^2$ and we can hope to close the
estimate. Again here, the scaling  helps us for the choice of the
Hilbert space~$\cH$. The scaling  of~$\cH$ must be the scaling
of~$H^{-\frac12}$. Moreover, because of  the operator
$\nablah\D_{\rm h}^{-1}$, it is natural to measure horizontal
derivatives and vertical derivatives differently. This leads to the
following definition.

\begin{defi}
\label{def2.1ad}
{\sl
For~$(s,s')$ in~$\R^2$, $H^{s,s'} $ denotes the space of tempered distribution~$a$  such~that
 $$
\|a\|^2_{H^{s,s'}} \eqdefa \int_{\R^3} |\xi_{\rm
h}|^{2s}|\xi_3|^{2s'} |\wh a (\xi)|^2d\xi <\infty\with \xi_{\rm h}=(\xi_1,\xi_2).
 $$
 For~$\theta$ in~$]0,1/2[$, we denote~$\cH_\theta\eqdefa H^{-\frac12+\theta,-\theta}$.
 }
 \end{defi}

Let us first remark that
\beq
\label{isoanisoinclud}
\begin{split}
(s,s')\in[0,\infty[^2 &\Longrightarrow \|a\|_{H^{s,s'}}\leq \|a\|_{H^{s+s'}}\andf\\
 (s,s')\in]-\infty,0]^2 &\Longrightarrow \|a\|_{H^{s+s'}}\leq \|a\|_{H^{s,s'}}.
 \end{split}
\eeq We want to emphasize  the fact that  anisotropy in the
regularity is highly related to the divergence free condition.
Indeed, let us consider a divergence free vector field~$w=(w^{\rm
h}, w^3)$ in~$H^{\frac12}$ and let us
estimate~$\|\partial_3w^3\|_{\cH_\th}$. By definition of
the~$\cH_\th$  norm, we have
 $$
 \|\partial_3w^3\|^2_{\cH_\th} = A_L+A_H \with  A_L\eqdefa \int_{|\xi_{\rm h}|\leq |\xi_3|}
 |\xi_{\rm h}|^{-1+2\th} |\xi_3|^{-2\th} |\cF(\partial_3w^3)(\xi)|^2 d\xi.
 $$
 In the case when~$|\xi_{\rm h}|\geq |\xi_3|$, we write that
 \beno
 A_H  \leq  \int_{\R^3} |\xi_3|\,|\wh w^3(\xi)|^2d\xi\leq \|w^3\|^2_{H^{\frac12}}.
 \eeno
 In the case when~$|\xi_{\rm h}|\leq |\xi_3|$, we use divergence free condition and write that   \beno
 A_L & \leq  &  \int_{|\xi_{\rm h}|\leq |\xi_3|}
 |\xi_{\rm h}|^{-1} |\cF(\dive_{\rm h} w^{\rm h})(\xi)|^2 d\xi\\
 & \leq &  \int_{\R^3}
 |\xi_{\rm h}|\, | \widehat{w}^{\rm h}(\xi)|^2 d\xi = \|w^{\rm h}\|_{H^{\frac12}}^2.
 \eeno
 Thus  for any divergence  free vector field~$w$ in~$H^{\frac12}$,  we have
 \beq
 \label{initialdataHtheta}
 \|\partial_3w^3\|_{\cH_\th} \leq C \|w\|_{H^{\frac12}}.
 \eeq

The first step of the proof of Theorem\refer{theoxplosaniscaling} is the following
 proposition:
  \begin{prop}
 \label{inegfondvroticity2D3D}
 {\sl
 Let us consider a solution $v$ of~$(NS)$ given by Theorem\refer{existenceomegaL3/2}. Then for $p$ in~$ ]4,6[,$
 a constant~$C$ exists such that  for any~$t<T^\star$
\beq \label{a.10qp}
\begin{split}
\frac 23  \bigl\| \,\om_{\frac34}(t)\bigr\|_{L^2}^{2}
 &+\frac 59 \int_0^t\bigl\|\nabla \om_{\f34}(t')\bigr\|_{L^2}^2\,dt'  \leq \biggl(\frac 23  \| \,| \om_0|^{\frac34}\|_{L^2}^{2}
\\
&\qquad{} + \Bigl(\int_0^t \|\p^2_3v^3(t')
\|^{2}_{\cH_\th}\,dt'\Bigr)^{\frac {3}{4}}\biggr)\exp
\Bigl(C\int_0^t \|v^3(t')\|_{\dH^{\frac 12+\frac2 p}}^p dt'\Bigr).
\end{split}
\eeq Here and in all that follows,  for scalar function~$a$ and
for~$\al$ in the interval~$]0,1[$, we always  denote
\beq\label{a.10wl} a_\al \eqdefa  \frac a {|a|} |a|^\al, \eeq so
that in particular $\om_{\frac34}=\om|\om|^{-\f14}$ and
$\om_{\f12}=\om|\om|^{-\f12}.$}
\end{prop}
Next we want to control~$\|\p^2_3v^3\|_{L^2_t(\cH_\theta)}$. As
already explained, a way to get rid of the only quadratic term
in~$\om$, namely
$$
-\partial_3^2\D^{-1}\Bigl( \sum_{\ell,m=1}^2 \partial_\ell v_{\rm curl}^m\partial_m v_{\rm curl}^\ell\Bigr).
$$
is to perform  an energy estimate for the norm~$\cH_\theta$.
\begin{prop}
 \label{estimadivhaniso}
 {\sl
 Let us consider a solution $v$ of~$(NS)$ given by Theorem\refer{existenceomegaL3/2}. For any~$p$ in~$]4,6[$ and $\th$ in
 $\bigl]\f12-\f2p,\f16\bigr[,$
 a constant~$C$ exists such that for any~$t<T^\star,$ we have
\beq
\label{b.4bqp}
\begin{split}
& \|\p_3v^3(t)\|_{\cH_\th}^2
 +\int_0^t\|\na\p_3v^3(t')\|_{\cH_\th}^2\,dt' \leq C \exp \Bigl(C\int_0^t
\|v^3(t')\|^{p}_{\dH^{\f12+\f2p}}\,dt' \Bigr)\\
 &\qquad\qquad{}\times\biggl(\|\Om_0\|_{L^{\f32}}^2
 + \int_0^t \Bigl(\|v^3(t')
\|_{\dH^{\f12+\f2p}}\bigl\|\om_{\f34}(t')\bigr\|_{L^2}^{2\bigl(\f13+\f1p\bigr)}\bigl\|\na\om_{\f34}(t')\bigr\|_{L^2}^{2\bigl(1-\f1p\bigr)}
\\
&\qquad\qquad\qquad\qquad\qquad{}+\|v^3(t')\|_{\dH^{\f12+\f2p}}^2\bigl\|\om_{\f34}(t')\bigr\|_{L^2}^{2\bigl(\f13+\f2p\bigr)}
\bigl\|\na\om_{\f34}(t')\bigr\|_{L^2}^{2\bigl(1-\f2p\bigr)}\Bigr)\,dt'\biggr).
\end{split}
\eeq
}
\end{prop}

As aforementioned observation, the non-linear terms of the equation
on~$\partial_3v^3$ contains  quadratic terms  in~$\om$.
  In spite of that, the terms in~$\om$ and in~$\partial_3v^3$ have the same homogeneity in
   \eqref{b.4bqp}. Let us point out that this is also the case in the estimate of
Proposition\refer{inegfondvroticity2D3D}. This will allow us to
close the estimates using Gronwall type arguments. More precisely,
we have the following proposition.
\begin{prop}
\label{SophisticatedGronwall}
{\sl
 Let $v$ be the unique solution of $(NS)$ given by Theorem
\ref{existenceomegaL3/2}. Then for any~$p$ in~$ ]4,6[$ and $\th$ in
 $\bigl]\f12-\f2p,\f16\bigr[,$ a
constant~$C$ exists such that, for any~$t<T^\ast$, we have \beno
\bigl\|\om_{\frac34}(t)\bigr\|_{L^2}^{2\bigl(\f{p+3}3\bigr)}+\bigl\|\na\om_{\f34}\bigr\|_{L^2_t(L^2)}^{2\bigl(\f{p+3}3\bigr)}
 & \leq  & C \|\Om_0\|_{L^{\f32}}^{\f{p+3}2}\cE(t)\andf\\
 \|\p_3v^3(t)\|_{\cH_\th}^2
 +\|\na\p_3v^3(t')\|_{L^2_t(\cH_\th)}^2
 & \leq & \|\Om_0\|_{L^{\f32}}^2 \cE(t) \with\\
 \cE(t) & \eqdefa &  \exp\biggl(C\exp
\Bigl(C\int_0^t\|v^3(t)\|_{\dH^{\f12+\f2p}}^p\,dt'\Bigr)\biggr).
\eeno
}
\end{prop}

The proof of this proposition from
Propositions\refer{inegfondvroticity2D3D} and\refer{estimadivhaniso}
is the purpose of Section\refer{Gronwall+}. It consists in plugging
the estimate of Proposition\refer{estimadivhaniso}
 into the one of  Proposition\refer{inegfondvroticity2D3D} and making
careful use of H\"older and convexity inequalities.

\medbreak

Now in order to conclude the proof of Theorem\refer{theoxplosaniscaling}, we need to prove that the control of
$$ \|\om\|_{L^\infty_t(L^{\f32})},\quad
 \int_0^t\|\nabla \om_{\f34}(t')\|_{L^2}^2\,dt'\,,\ \int_0^t \|\p^2_3v^3(t')
\|^{2}_{\cH_\th}\,dt' \andf
\int_0^t\|v^3(t')\|_{\dH^{\frac12+\frac2p}}^pdt'
$$
prevents the solution $v$ of $(NS)$ from  blowing up. This will be
done through the following blow up result, which may have its own
interest. Before stating the theorem, let us recall the definition
of some class of Besov spaces.

\begin{defi}
\label{def2.2}
{\sl If~$\s$ is a positive real number, we define the space~$B^{-\s}_{\infty,\infty}$ as the space of
tempered distributions~$f$  such that
$$
\|f\|_{B^{-\s}_{\infty,\infty}} \eqdefa \sup_{t>0} t^{\frac \s 2} \|e^{t\D}f\|_{L^\infty} <\infty.
$$
For~$p$ in~$]1,\infty[$,  we shall use the notation~$\cB_p\eqdefa B^{-2+\frac 2p}_{\infty,\infty}$.
}
\end{defi}
These spaces are in some sense the largest ones which have  a fixed
scaling. Indeed,  let us consider any Banach space~$E$ which can be
continuously embedded  into the space of tempered
distribution~$\cS'(\R^3)$ such that
$$
\forall (\lam ,\vect a)\in\, ]0,\infty[\times \R^3\,,\ \lam^\s \|f(\lam\cdot+\vect a)\|_E\sim \|f\|_{E}.
$$
The first hypothesis on~$E$ implies that   a constant~$C$ exists such that
$$
\langle f,e^{-|\cdot|^2}\rangle  \leq C\|f\|_{E}.
$$
The scaling hypothesis on~$E$ implies, after a change of variables
in the left-hand side of the above inequality, that
$$
\forall t \in ]0,\infty[\,,\ t^{\frac \s 2} \|e^{t\D}f\|_{L^\infty} \leq C \|f\|_E.
$$
Then the following theorem can be understood as  an end point blow
up theorem  for the incompressible Navier-Stokes equation.

\begin{thm}
\label{blowupBesovendpoint}
{\sl
 Let $v$ be a solution of~$(NS)$  in the
space~$C([0,T^\star[;H^{\frac12})\cap
L^2_{\rm loc}([0,T^\star[;H^{\frac32})$. If~$T^\star$ is the
maximal time of existence and $T^\ast<\infty,$ then for any~$(p_{k,\ell})$ in~$]1,\infty[^9$, one has
$$
\sum_{1\leq k,\ell\leq3} \int_0^{T^\star} \|\partial_\ell
v^k(t)\|^{p_{k,\ell}}_{\cB_{p_{k,\ell}}} dt=\infty.
$$
}
\end{thm}

 It is easy to observe that
\beq \label{inclusionBesov} \|\partial_\ell v\|_{\cB_p} \lesssim
\|v\|_{L^q} \with \frac2p+\frac3q=1\,, \, p>2,\quad\mbox{and}\quad
L^q \subset \cB_p\with \frac2p+\frac3q=2\, \cdotp \eeq In
particular, Theorem\refer{blowupBesovendpoint} implies blow up
criteria\refeq{blowupFond} and~\eqref{blowupBeiraVega}. It
generalizes also the result by  D. Fang and C. Qian (see~\cite{FQ})
who  proved sort of combined version of blow up
criteria\refeq{blowupFond} and~\eqref{blowupBeiraVega}, like for
instance critical Lebesgue norms of horizontal components of the
vorticity and of derivative to the third component of the velocity.

The paper is organized as follows. In the third section, we explain how the~$L^{\frac 32}$ energy estimate on
 the vorticity allows to prove the local existence of a solution to
 $(NS)$
  which satisfies the smoothing effect~``$\nabla |\Om|^{\frac34}$ belongs to~$L^2_{\rm loc} ([0,T^\star[;L^2(\R^3)$".

In the fourth section, we present the tool of anisotropic
Littlewood-Paley theory and some properties of anisotropic Besov
spaces which play a key role in the proof of
Propositions\refer{inegfondvroticity2D3D}
and\refer{estimadivhaniso}.

In the fifth section, we prove
Proposition\refer{inegfondvroticity2D3D}.

In the sixth section, we prove Proposition\refer{estimadivhaniso}.

In the seventh section,  we explain how to deduce Proposition\refer{SophisticatedGronwall} from  Propositions\refer{inegfondvroticity2D3D} and\refer{estimadivhaniso}.

In the eighth section, we prove Theorem\refer{blowupBesovendpoint}
and conclude the proof of Theorem \ref{theoxplosaniscaling}.

\medbreak Before going on, let us introduce some notations that will
be used in all that follows. Let $A, B$ be two operators, we denote
$[A;B]=AB-BA,$ the commutator between $A$ and $B$. For~$a\lesssim
b$, we mean that there is a uniform constant $C,$ which may be
different on different lines, such that $a\leq Cb$.  We denote by
$(a|b)_{L^2}$ the $L^2(\R^3)$ inner product of $a$ and $b$. For $X$
a Banach space and $I$ an interval of $\R,$ we denote by $C(I;\,X)$
the set of continuous functions on~$I$ with values in $X.$    For
$q$ in~$[1,+\infty],$ the notation $L^q(I;\,X)$ stands for the set
of measurable functions on $I$ with values in $X,$ such that
$t\longmapsto\|f(t)\|_{X}$ belongs to $L^q(I).$ Finally, we denote
by $L^p_T(L^q_{\rm h}(L^r_{\rm v}))$ the space $L^p([0,T];
L^q(\R_{x_{\rm h}};L^r(\R_{x_3})))$ with $x_{\rm h}=(x_1,x_2),$ and
$\nabla_{\rm h}=(\p_{x_1},\p_{x_2}),$ $\D_{\rm
h}=\p_{x_1}^2+\p_{x_2}^2.$

\setcounter{equation}{0}
\section{The local wellposedness of $(NS)$ for vorticity in~$L^{\frac32}$ revisited}

By virtue of  $(NS)$, the vorticity $\Om=\na\times v$ satisfies the
equation
$$
(NSV)\quad\left\{
\begin{array}{c}
\partial_t\Om -\D\Om+v\cdot\nabla \Om -\Om\cdot\nabla v=0\\
\Om_{|t=0} =\Om_0.
\end{array}
\right.
$$
The key  ingredient to prove Theorem \ref{existenceomegaL3/2}  is
the following lemma:
\begin{lem}
\label{estimpropagationLp}
{\sl
Let~$p$ be in~$]1,2[$ and~$a_0$ a
function in~$L^p$. Let us consider a function~$f$
in~$L^1_{\rm loc}(\R^+;L^p)$ and~$v$ a divergence free vector field
in~$L^2_{\rm loc}(\R^+;L^\infty)$ . If $a$ solves
$$
(T_v)\quad\left\{
\begin{array}{c} \partial_t a  -\D a +v\cdot\nabla a=f\\
a_{|t=0}=a_0
\end{array}
\right.
$$
then~$a$ satisfies ~$|a|^{\frac p2}$ belonging
to~$L^\infty_{\rm loc}(\R^+;L^2)\cap L^2_{\rm loc}(\R^+;\dH^1)$ and
\beq\label{tvestimate} \begin{split} \frac1p\int_{\R^3} | a(t,x)
|^pdx &+(p-1) \int_0^t \int_{\R^3} |\nabla a(t',x)|^2|a(t',x)|^{p-2}
dx\,dt'\\
&=\frac1p\int_{\R^3} | a_0(x) |^pdx+ \int_0^t\int_{\R^3} f(t',x)
a(t',x) |a(t',x)|^{p-2}dx\,dt'. \end{split} \eeq
}
\end{lem}
\begin{proof}  Note that for $p$ in~$ ]1,2[,$ $\na a=\na
a|a|^{\f{p-2}2}\times |a|^{\f{2-p}2},$ which belongs to
$L^2_{\rm loc}(\R^+; L^p)$ according to the energy inequality
\eqref{tvestimate}. Moreover,  $v$ belongs to~$L^2_{\rm loc}(\R^+;L^\infty),$ so that
$v\cdot\na a$ is in~$ L^1_{\rm loc}(\R^+; L^p),$ and hence arguing by
density, we can assume that all the functions in $(T_v)$ are smooth.
As the function~$r\mapsto r^p$ is~ $C^1$, we first write that
\beno
\frac1p\frac d{dt} \int_{\R^3}|a(t,x)|^p dx   & = & \int_{\R^3} \partial_t a\, a\, |a|^{p-2} dx\\
& =& -\frac 1p\int_{\R^3} v(t,x)\cdot \nabla |a |^p(t,x) dx
+\int_{\R^3} \D a(t,x) a (t,x)|a(t,x)|^{p-2}dx\\
 & &\qquad\qquad\qquad\qquad\qquad\qquad\qquad{}+\int_{\R^3} f(t,x)a
(t,x)|a(t,x)|^{p-2}dx.
\eeno
 As~$v$ is assumed to be divergence free, we get
$$
\frac1p\frac d{dt} \int_{\R^3}|a(t,x)|^p dx =\int_{\R^3} \D a(t,x) a (t,x)|a(t,x)|^{p-2}dx{}+\int_{\R^3} f(t,x)a (t,x)|a(t,x)|^{p-2}dx.
$$
Integrating the above inequality over $[0,t]$ yields
\beq
\label{estimpropagationLpdemoeq1}
\begin{split}
\frac1p\int_{\R^3}|a(t,x)|^p dx &=\frac1p\int_{\R^3}|a_0(x)|^p dx+\int_0^t\int_{\R^3} \D a(t',x) a (t',x)|a(t',x)|^{p-2}dxdt'\\
&\qquad\qquad\qquad{}+\int_0^t\int_{\R^3} f(t',x)a
(t',x)|a(t',x)|^{p-2}dx.
\end{split}
\eeq
In the case when~$p\geq 2$,  the function~$r\mapsto r^{p-1}$ is~$C^1$ and then an integration by parts implies that
$$
\int_{\R^3} \D a(t,x) a (t,x)|a(t,x)|^{p-2}dx = -(p-1) \int_{\R^3} |\nabla a (t,x)|^2|a(t,x)|^{p-2}dx.
$$
In the case when~$p$ is less than~$2$, some regularization has to be
made.  Indeed, even for smooth function, the fact that~$|a|^{\frac
p2}$ belongs to~$\dH^1$ is not obvious.  As~$a$ is supposed to be
smooth, in particular, we have that $a$ is bounded and~$\D
a\,|a|^{p-1} $ belongs to~$L^\infty_{\rm loc}(\R^+,L^1)$. Thus, using
Lebesgue's theorem, we infer that \beq
\label{estimpropagationLpdemoeq2}
\begin{split}
&\lim_{\d\rightarrow 0}  \int_0^t\int_{\R^3} \D a(t',x) a (t',x)\bigl(|a(t',x)|+\d\bigr)^{p-2}\,dx\,dt'\\
&\qquad\qquad\qquad\qquad\qquad\qquad{}=\int_0^t\int_{\R^3} \D
a(t',x) a (t',x)|a(t',x)|^{p-2}\,dx\,dt' .
\end{split}
\eeq As the function~$r\mapsto (r+\d)^{p-2}$ is smooth for any
positive~$\d$, we obtain
$$
\longformule{ -\int_{\R^3} \D a(t',x) a
(t',x)\bigl(|a(t',x)|+\d\bigr)^{p-2}\,dx = \int_{\R^3} |\nabla
a(t',x)|^2 \bigl(|a(t',x)|+\d\bigr)^{p-2}\,dx } {{} +(p-2)
\int_{\R^3} \nabla a(t',x) \cdot (\nabla |a|)(t',x)  a(t',x)
\bigl(|a(t',x)|+\d\bigr)^{p-3}dx. }
$$
It is well-known that
$$
\nabla |a| =\nabla a\, \frac {a}{|a|}\,\cdotp
$$
Thus we get  by time integration that
$$
\longformule{ \int_0^t\int_{\R^3} \D a(t',x) a
(t',x)\bigl(|a(t',x)|+\d\bigr)^{p-2}dx = \int_0^t\int_{\R^3} |\nabla
a(t',x)|^2 \bigl(|a(t',x)|+\d\bigr)^{p-2}\,dx\,dt' } {{} +(p-2)
\int_0^t\int_{\R^3} | \nabla a(t',x) |^2|  a(t',x)
|\bigl(|a(t',x)|+\d\bigr)^{p-3}\,dx\,dt'. }
$$
For the term in the right-hand side of the above inequality, thanks
to\refeq{estimpropagationLpdemoeq2} and to the monotonic convergence
theorem, we get that~$|\nabla a|^2|a|^{p-2}$ belongs
to~$L^1_{\rm loc}(\R^+;L^1)$ and that
$$
-\int_0^t\int_{\R^3} \D a(t',x) a (t',x)|a(t',x)|^{p-2}\,dx\,dt'
=(p-1) \int_0^t\int_{\R^3} |\nabla a(t',x)|^2 |a(t',x)|^{p-2}
\,dx\,dt'.
$$ Resuming the above estimate into
\eqref{estimpropagationLpdemoeq1} leads to \eqref{tvestimate}. This
proves the lemma.
\end{proof}

We remark that we shall use Lemma\refer{estimpropagationLp} in the
case when~$p=3/2$. Indeed, by virtue of \eqref{a.10wl}, one has
$$
|\nabla a_{\frac34} | =\bigl|\nabla | a |^{\frac34} \bigr| =  \frac
34 |\nabla a |\,|a|^{-\frac14}.
$$
Then \eqref{tvestimate} applied for $p=\f32$ gives rise to
\beq\label{klips} \frac 23 \bigl\|a_{\frac34}(t) \bigr\|_{L^2}^2
+\frac 8 9 \int_0^t \bigl\|\nabla a_{\frac34}(t')\bigr\|_{L^2}^2 dt'
=\frac 23 \bigl\||a_0|^{\f34}\bigr\|_{L^2}^2+\int_0^t \int_{\R^3}
f(t',x)a_{\frac12}(t',x)\,dx\, dt'. \eeq

Let us turn to the proof of Theorem \ref{existenceomegaL3/2}.

 \no\begin{proof}[Proof of
Theorem \ref{existenceomegaL3/2}] Biot-Sarvart's law  claims that
$v_0=-\na\D^{-1}\times \Om_0$. This implies
that~$\|v_0\|_{H^{\frac12} } \lesssim \|\Om_0\|_{H^{-\frac12}}$.
Using the dual Sobolev embedding~$\|f\|_{H^{-\frac12}}\lesssim
\|f\|_{L^{\frac 32}}$, we deduce that ~$v_0$ belongs
to~$H^{\frac12}$.  Then applying Fujita-Kato theory
\cite{fujitakato} ensures that $(NS)$ has a unique solution $v$ on
$[0,T^\ast[$ in the space~$C([0,T^\ast[;H ^{\f12})\cap L^2_{\rm
loc}([0,T^\star[;H ^{\frac32}).$ Moreover, it follows from
Proposition B.1 of \cite{CGZ} that~$v$ belongs to~$L^2_{\rm
loc}([0,T^\ast[; L^\infty).$ Then to apply Lemma
\ref{estimpropagationLp} for~$(NSV)$ with the external
force~$f=\Om\cdot\nabla v$, we only need to estimate this term.
Indeed as the solution~$v$ belongs to~$L^2_{\rm loc}([0,T^\star[;H
^{\frac32})$, we use   Sobolev inequality to get \beno
\Bigl|\int_0^t\int_{\R^3}\Om\cdot\na
v\,\Om|\Om|^{-\f12}\,dx\,dt'\Bigr| & \leq &\int_0^t
\|\Om\cdot\nabla v(t') \|_{L^{\frac32}}\|\Om(t')\|_{L^{\f32}}^{\f12}\,dt'\\
& \leq&\int_0^t\|\Om(t')\|_{L^3}\|\nabla v(t')\|_{L^3}\|\Om(t')\|_{L^{\f32}}^{\f12}\,dt'\\
& \leq&C\int_0^t\|\nabla v(t')\|_{\dH^{\frac
12}}^2\|\Om(t')\|_{L^{\f32}}^{\f12}\,dt'.
 \eeno
  By virtue of $(NSV)$,  by applying Lemma
\ref{estimpropagationLp} and using the convexity inequality
  $$
  ab\leq \frac 23 a^{\frac32}+\frac 13 b^3\with a = \|\nabla v(t')\|_{H^{\frac12}}^{\frac43}\andf
  b =  \|\nabla v(t')\|_{H^{\frac12}}^{\frac23}  \|\Om(t')\|_{L^{\f32}}^{\f12},
  $$
   we infer that
   $$ \longformule{
\f23\int_{\R^3}|\Om(t,x)|^{\f32}\,dx+\f12\int_0^t\int_{\R^3}|\na\Om(t',x)|^2|\Om(t',x)|^{-\f12}\,dx\,dt'}{{}\leq
\f23\int_{\R^3}|\Om_0(x)|^{\f32}\,dx+\int_0^t\|\na
v(t')\|_{\dH^{\f12}}^2\,dt'+C\int_0^t\|\na
v(t')\|_{\dH^{\f12}}^2\|\Om(t')\|_{L^{\f32}}^{\f32}\,dt'.} $$
Applying Gronwall Lemma gives rise to $$ \longformule{
\f23\int_{\R^3}|\Om(t,x)|^{\f32}\,dx+\f12\int_0^t\int_{\R^3}|\na\Om(t',x)|^2|\Om(t',x)|^{-\f12}\,dx\,dt'}{{}
\leq \Bigl(\f23\int_{\R^3}|\Om_0(x)|^{\f32}\,dx+\int_0^t\|\na
v(t')\|_{\dH^{\f12}}^2\,dt'\Bigr)\exp\Bigl(C\int_0^t\|\na
v(t')\|_{\dH^{\f12}}^2\,dt'\Bigr).} $$ Thus
Theorem\refer{existenceomegaL3/2} is proved.
 \end{proof}

\medbreak As a conclusion of this section,  let us establish some
Sobolev type inequalities which involves the regularities of
~$a_{\frac34}$ and~$\nabla a_{\frac34}$ in~$L^2$.
\begin{lem}
\label{BiotSavartomega}
{\sl
We have
\beq
\label{estimbasomega34}
 \|\nabla a\|_{L^{\frac32}} \lesssim \bigl\|\nabla a_{\frac34}\bigr\|_{L^2}
 \bigl\|a_{\frac34}  \bigr\|_{L^2}^{\frac13}.
 \eeq
Moreover, for~$s$  in $\ds\bigl[-1/2\,\virgp\, 5/6 \bigr]$, we have
\beq \label{estimbasomega34.0} \|a\|_{H^s} \leq C
\|a_{\frac34}\|_{L^2} ^{\frac 56-s}  \|\nabla a_{\frac34}\|_{L^2}
^{\frac 12+s} . \eeq }
\end{lem}
\begin{proof}
Notice that due to \eqref{a.10wl}, $\ds |\na a |=\f43|\na
a_{\f34}|\,|a|^{\f14},$ then we get \eqref{estimbasomega34} by using
H\"older inequality. The dual Sobolev inequality claims that \beq
\label{BiotSavartomegademoeq1} \|a\|_{H^{-\frac12}} \leq C
\|a\|_{L^{\frac32}}= C \|a_{\frac34}\|_{L^2} ^{\frac 43}. \eeq
Moreover, using again that $|\nabla a|= \f43|\nabla
a_{\frac34}|\,|a|^{\frac14}$, H\"older inequality implies that \beno
\|\nabla a \|_{L^{\frac95}} & \leq &  \f43\|\nabla a_{\frac34}\|_{L^2} \|\,|a|^{\frac14}\|_{L^{18}}\\
& \leq & \f43 \|\nabla a_{\frac34}\|_{L^2}
\|a_{\frac34}\|_{L^6}^{\frac13}. \eeno Sobolew embedding of~$H^1$
into~$L^6$ then ensures that \beq \label{BiotSavartomegademoeq2}
\|\nabla a \|_{L^{\frac95}}\leq C  \|\nabla
a_{\frac34}\|_{L^2}^{\frac43}. \eeq Sobolev embedding
of~$W^{1,\frac95}$ into~$H^{\frac 56}$ leads to
$$
\|a\|_{H^{\frac56}} \leq C  \|\nabla a_{\frac34}\|_{L^2}^{\frac43},
$$
from which and \eqref{BiotSavartomegademoeq1}, we concludes the
proof of \eqref{estimbasomega34.0} and hence the lemma by using
interpolation inequality between~$H^s$ Sobolev spaces.
\end{proof}

\setcounter{equation}{0} \section{Some estimates related to
Littlewood-Paley analysis}

As we shall use the anisotropic Littlewood-Paley theory, we recall
the functional space framework we are going to use in this section.
As in\ccite{CDGG}, \ccite{CZ1} and \ccite{Pa02}, the definitions of
the spaces we are going to work with requires anisotropic dyadic
decomposition   of the Fourier variables. Let us recall from
\cite{BCD} that \beq
\begin{split}
&\Delta_k^{\rm h}a=\cF^{-1}(\varphi(2^{-k}|\xi_{\rm h}|)\widehat{a}),\qquad
\Delta_\ell^{\rm v}a =\cF^{-1}(\varphi(2^{-\ell}|\xi_3|)\widehat{a}),\\
&S^{\rm h}_ka=\cF^{-1}(\chi(2^{-k}|\xi_{\rm h}|)\widehat{a}),
\qquad\ S^{\rm v}_\ell a =
\cF^{-1}(\chi(2^{-\ell}|\xi_3|)\widehat{a})
 \quad\mbox{and}\\
&\Delta_ja=\cF^{-1}(\varphi(2^{-j}|\xi|)\widehat{a}),
 \qquad\ \
S_ja=\cF^{-1}(\chi(2^{-j}|\xi|)\widehat{a}), \end{split}
\label{1.0}\eeq where $\xi_{\rm h}=(\xi_1,\xi_2),$ $\cF a$ and
$\widehat{a}$ denote the Fourier transform of the distribution $a,$
$\chi(\tau)$ and~$\varphi(\tau)$ are smooth functions such that
 \beno
&&\Supp \varphi \subset \Bigl\{\tau \in \R\,/\  \ \frac34 \leq
|\tau| \leq \frac83 \Bigr\}\andf \  \ \forall
 \tau>0\,,\ \sum_{j\in\Z}\varphi(2^{-j}\tau)=1,\\
&&\Supp \chi \subset \Bigl\{\tau \in \R\,/\  \ \ |\tau|  \leq
\frac43 \Bigr\}\quad \ \ \andf \  \ \, \chi(\tau)+ \sum_{j\geq
0}\varphi(2^{-j}\tau)=1.
 \eeno

\begin{defi}\label{def4.1}
{\sl  Let $(p,r)$ be in~$[1,+\infty]^2$ and~$s$ in~$\R$. Let us consider~$u$ in~${\mathcal
S}_h'(\R^3),$ which means that $u$ is in~$\cS'(\R^3)$ and satisfies~$\ds\lim_{j\to-\infty}\|S_ju\|_{L^\infty}=0$. We set
$$
\|u\|_{B^s_{p,r}}\eqdefa\big\|\big(2^{js}\|\Delta_j u\|_{L^{p}}\big)_j\bigr\|_{\ell
^{r}(\ZZ)}.
$$
\begin{itemize}

\item
For $s<\frac{3}{p}$ (or $s=\frac{3}{p}$ if $r=1$), we define $
B^s_{p,r}(\R^3)\eqdefa \big\{u\in{\mathcal S}_h'(\R^3)\;\big|\; \|
u\|_{B^s_{p,r}}<\infty\big\}.$

\item
If $k$ is  a positive integer and if~$\frac{3}{p}+k\leq s<\frac{3}{p}+k+1$ (or
$s=\frac{3}{p}+k+1$ if $r=1$), then we define~$ B^s_{p,r}(\R^3)$  as
the subset of distributions $u$ in~${\mathcal S}_h'(\R^3)$ such that
$\partial^\beta u$ belongs to~$ B^{s-k}_{p,r}(\R^3)$ whenever $|\beta|=k.$
\end{itemize}
}
\end{defi}
Let us remark that in the particular case when $p=r=2,$ $B^s_{p,r}$
coincides with the classical homogeneous Sobolev spaces $\dH^s$.
Moreover, in the case when~$p=r=\infty$, it coincides with the
spaces defined in Definition\refer{def2.1ad} (see for instance
Theorem~2.34 on page~76 of\ccite{BCD}).


\medbreak

Similar to Definition \ref{def4.1}, we can also define the
homogeneous anisotropic Besov space.

\begin{defi}\label{anibesov}
{\sl Let us define the space $\bigl(B^{s_1}_{p,q_1}\bigr)_{\rm
h}\bigl(B^{s_2}_{p,q_2}\bigr)_{\rm v}$ as the space of distribution
in~$\cS'_h$  such that
$$
\|u\|_{\bigl(B^{s_1}_{p,q_1}\bigr)_{\rm
h}\bigl(B^{s_2}_{p,q_2}\bigr)_{\rm v}}\eqdefa \biggl(\sum_{k\in\Z}
2^{q_1ks_1} \Bigl(\sum_{\ell\in\Z}2^{q_2\ell s_2}\|\D_k^{\rm
h}\D_\ell^{\rm
v}u\|_{L^p}^{q_2}\Bigr)^{\f{q_1}{q_2}}\biggr)^{\f1{q_1}}
$$
is finite.
}
\end{defi}
We remark that when $p=q_1=q_2=2,$ the anisotropic Besov space
$\bigl(B^{s_1}_{p,q_1}\bigr)_{\rm h}\bigl(B^{s_2}_{p,q_2}\bigr)_{\rm
v}$ coincides with the classical homogeneous anisotropic Sobolev
space $\dH^{s_1,s_2}$ and thus the space~$\bigl(B^{-\frac
12+\theta}_{2,2}\bigr)_{\rm h}\bigl(B^{-\theta}_{2,2}\bigr)_{\rm v}$
is the space~$\cH_\theta$ defined in Definition\refer{def2.1ad}. Let
us also remark that in the case when
 $q_1$ is different from~$q_2$, the order of summation is important.

\medbreak
 For the
convenience of the readers, we recall the following anisotropic
Bernstein type lemma from \cite{CZ1, Pa02}:

\begin{lem}
\label{lemBern} {\sl Let $\cB_{h}$ (resp.~$\cB_{v}$) a ball
of~$\R^2_{h}$ (resp.~$\R_{v}$), and~$\cC_{h}$ (resp.~$\cC_{v}$) a
ring of~$\R^2_{h}$ (resp.~$\R_{v}$); let~$1\leq p_2\leq p_1\leq
\infty$ and ~$1\leq q_2\leq q_1\leq \infty.$ Then there holds:

\smallbreak\noindent If the support of~$\wh a$ is included
in~$2^k\cB_{h}$, then
\[
\|\partial_{x_{\rm h}}^\alpha a\|_{L^{p_1}_{\rm h}(L^{q_1}_{\rm v})}
\lesssim
2^{k\left(|\al|+2\left(\frac1{p_2}-\frac1{p_1}\right)\right)}
\|a\|_{L^{p_2}_{\rm h}(L^{q_1}_{\rm v})}.
\]
If the support of~$\wh a$ is included in~$2^\ell\cB_{v}$, then
\[
\|\partial_{3}^\beta a\|_{L^{p_1}_{\rm h}(L^{q_1}_{\rm v})} \lesssim
2^{\ell(\beta+(\frac1{q_2}-\frac1{q_1}))} \| a\|_{L^{p_1}_{\rm
h}(L^{q_2}_{\rm v})}.
\]
If the support of~$\wh a$ is included in~$2^k\cC_{h}$, then
\[
\|a\|_{L^{p_1}_{\rm h}(L^{q_1}_{\rm v})} \lesssim
2^{-kN}\sup_{|\al|=N} \|\partial_{x_{\rm h}}^\al a\|_{L^{p_1}_{\rm
h}(L^{q_1}_{\rm v})}.
\]
If the support of~$\wh a$ is included in~$2^\ell\cC_{v}$, then
\[
\|a\|_{L^{p_1}_{\rm h}(L^{q_1}_{\rm v})} \lesssim 2^{-\ell N}
\|\partial_{x_3}^N a\|_{L^{p_1}_{\rm h}(L^{q_1}_{\rm v})}.
\]
}
\end{lem}

\medbreak

As a corollary of Lemma \ref{lemBern}, we have the following
inequality, if~$1\leq p_2\leq p_1$, \beq
\label{inclusionSobolevtypeaniso}
\|a\|_{\bigl(B^{s_1-2\left(\frac1{p_2}-\frac 1
{p_1}\right)}_{p_1,q_1}\bigr)_{\rm
h}\bigl(B^{s_2-\left(\frac1{p_2}-\frac 1
{p_1}\right)}_{p_1,q_2}\bigr)_{\rm v}}\lesssim
\|a\|_{\bigl(B^{s_1}_{p_2,q_1}\bigr)_{\rm
h}\bigl(B^{s_2}_{p_2,q_2}\bigr)_{\rm v}}. \eeq

\medbreak

 To consider the
product of a distribution in the isentropic Besov space with a
distribution in the anisotropic Besov space, we need the following
result which allows to embed  isotropic Besov spaces into the
anisotropic ones.
\begin{lem}\label{embeda}
{\sl
Let $s$ be a positive real number and $(p,q)$ in~$ [1,\infty]$ with
$p\geq q.$ Then one has \beno
\|a\|_{L^p_{\rm h}\bigl((B^s_{p,q})_{\rm v}\bigr)}\lesssim \|a\|_{B^s_{p,q}}.
\eeno
}
\end{lem}
\begin{proof} Once noticed that,  an integer~$N_0$ exists such that, if~$j$ is less or equal to~$\ell-N_0$
 then the operator~$\D^{\rm v}_\ell\D_j$ is identically~$0$, we can write that
\beno 2^{\ell s}\|\D_\ell^{\rm v}a\|_{L^p} & \lesssim & 2^{\ell
s}\sum_{\ell\leq j+N_0}\|\D_\ell^{\rm v}\D_j
a\|_{L^p}\\
& \lesssim&\sum_{\ell\leq j+N_0} 2^{(\ell-j)s}2^{j s}\|\D_j
a\|_{L^p}.
\eeno
Because~$s$ is positive,  Young inequality on~$\ZZ$ implies that
$$
\bigl\|\bigl( 2^{\ell s} \|\D_\ell^{\rm
v}a\|_{L^p}\bigr)_\ell\bigr\|_{\ell^q(\ZZ)} \lesssim
\|a\|_{B^s_{p,q}}.
$$
Due to $p\geq q,$ Minkowski inequality implies that \beno
\|a\|_{L^p_{\rm h}\bigl((B^s_{p,q})_{\rm v}\bigr)} &= &
\bigl\|\bigl(2^{\ell s}\|\D_\ell^{\rm v}
a(x_{\rm h},\cdot)\|_{L^p_{\rm v}}\bigr)_\ell\bigr\|_{\ell^q(\ZZ)}\bigr\|_{L^p_{\rm h}}\\
&\lesssim & \bigl\|\bigl( 2^{\ell s} \|\D_\ell^{\rm v}a\|_{L^p}\bigr)_\ell\bigr\|_{\ell^q(\ZZ)}\\
& \lesssim & \|a\|_{B^s_{p,q}}.
 \eeno
The lemma is proved.
\end{proof}

\begin{lem}
\label{isoaniso}
 {\sl For any~$s$ positive and any~$\theta$ in~$]0,s[$,
we have
$$
\|f\|_{(B^{s-\theta}_{p,q})_{\rm h}(B^{\theta}_{p,1})_{\rm v}} \lesssim
\|f\|_{B^{s}_{p,q}}.
$$
}
\end{lem}
\begin{proof} This lemma means exactly that
 \beq
\label{inclusionanisoBesovell1demoeq1} V_k \eqdefa  \sum_{\ell\in\Z}
2^{\ell\theta}\|\D_k^{\rm h}\D_\ell^{\rm v}f\|_{L^p} \lesssim
c_{k,q} 2^{-k(s-\theta)}\|f\|_{B^s_{p,q}} \with( c_{k,q})_k\in
\ell^q(\ZZ).
 \eeq
We distinguish the case when~$\ell$ is less or equal to~$k$ from the
case when~$\ell$ is greater than~$k$. Using the fact that the
operators~$\D_\ell^{\rm v}$ are uniformly bounded on~$L^p,$ we write
 \ben
2^{k(s-\theta)} V_k &= &2^{k(s-\theta)} \sum_{\ell\leq k}
2^{\ell\theta} \|\D_k^{\rm h}\D_\ell^{\rm v}f\|_{L^p}
 +2^{k(s-\theta)} \sum_{\ell> k} 2^{\ell\theta}  \|\D_k^{\rm h}\D_\ell^{\rm v}f\|_{L^p}\nonumber\\
 \label{inclusionanisoBesovell1demoeq2}
  &\lesssim  & 2^{ks} \|\D_k^{\rm h}f\|_{L^p} +2^{k(s-\theta)} \sum_{\ell> k}
2^{\ell\theta} \|\D_k^{\rm h}\D_\ell^{\rm v}f\|_{L^p}. \een In the
case when~$\ell$ is greater than~$k$,  the set~$2^k\cC_{\rm h}\times
2^\ell\cC_{\rm v}$ is included a  ring of the type~$2^\ell\wt\cC$.
Thus, if~$|j-\ell|$ is greater than some fixed integer~$N_0$, then
we have~$\D_{j}\D_k^{\rm h}\D_\ell^{\rm v}\equiv 0$.  This gives
$$
\sum_{\ell> k} 2^{\ell\theta} \|\D_k^{\rm h}\D_\ell^{\rm v}f\|_{L^p}
\lesssim \sum_{\substack{|j-\ell|\leq N_0\\\ell>k}}
2^{\ell\theta} \|\D_{j}\D_k^{\rm h}\D_\ell^{\rm v}f\|_{L^p}.
$$Then using again that the
operators~$\D_\ell^{\rm v}$ and~$\D_k^{\rm h}$ are uniformly bounded
on~$L^p,$ we infer that
$$
\sum_{\ell> k} 2^{\ell\theta} \|\D_k^{\rm h}\D_\ell^{\rm v}f\|_{L^p}
\lesssim\sum_{j>k-N_0}2^{-j(s-\th)}2^{js}\|\D_{j}f\|_{L^p}.
$$
Moreover, we have~$\D_j\D_k^{\rm h}=0$ if~$j\leq k-N_1$. We thus
deduce from\refeq{inclusionanisoBesovell1demoeq2} that
$$
2^{k(s-\theta)} V_k \lesssim \sum_{j\geq k-N_1} 2^{-(j-k)s}
2^{js}\|\D_{j} f\|_{L^p}+ \sum_{j\geq k-N_0} 2^{-(j-k)(s-\theta)}
2^{js}\|\D_{j} f\|_{L^p}.
$$
This gives\refeq{inclusionanisoBesovell1demoeq1} and thus the lemma.
\end{proof}

One of the main motivation of using anisotropic Besov space is the
proof of the following proposition.
\begin{prop}
\label{BiotSavartBesovaniso}
{\sl
 Let~$v$ be a divergence free vector
field. Let us consider~$(\al,\theta)$ in~$]0,1/2[^2.$
Then we have
$$
\|v^{\rm h}\|_{\bigl(B^{1}_{2,1}\bigr)_{\rm
h}\bigl(B^{\f12-\al}_{2,1}\bigr)_{\rm v}}\lesssim \bigl\|\,
\om_{\frac34}\bigr\|_{L^2} ^{\frac13+\al} \bigl\|\nabla
\om_{\frac34}\bigr\|_{L^2} ^{1-\al} +\|\partial_3v^3\|_{L^2}^\al
\|\nabla
\partial_3v^3\|_{\cH_\theta}^{1-\al}.
$$
}
\end{prop}
\begin{proof}
Using horizontal Biot-Savart law\refeq {a.1} and
Lemma\refer{lemBern}, we have \beq
\label{BiotSavartBesovanisodemoeq1}\|v^{\rm
h}\|_{\bigl(B^{1}_{2,1}\bigr)_{\rm
h}\bigl(B^{\f12-\al}_{2,1}\bigr)_{\rm v}}\lesssim
\|\om\|_{\bigl(B^{0}_{2,1}\bigr)_{\rm
h}\bigl(B^{\f12-\al}_{2,1}\bigr)_{\rm v}}+
\|\partial_3v^3\|_{\bigl(B^{0}_{2,1}\bigr)_{\rm
h}\bigl(B^{\f12-\al}_{2,1}\bigr)_{\rm v}}. \eeq Applying
Lemmas\refer{lemBern} and Lemma \refer{isoaniso} gives \ben
\|\om\|_{\bigl(B^{0}_{2,1}\bigr)_{\rm
h}\bigl(B^{\f12-\al}_{2,1}\bigr)_{\rm v}} & \lesssim &
\|\om\|_{\bigl(B^{\frac19}_{\frac95,1}\bigr)_{\rm
h}\bigl(B^{\f59-\al}_{\frac 95,1}\bigr)_{\rm v}}
\nonumber\\
\label{BiotSavartBesovanisodemoeq2}& \lesssim & \|\om\|_{B^{\frac
23-\al}_{\frac95,1}}. \een Now let us
estimate~$\|\om\|_{B^s_{\f95,1}}$  in terms
of~$\bigl\|\om_{\frac34}\bigr\|_{L^2}$
and~$\bigl\|\na\om_{\frac34}\bigr\|_{L^{2}}$.
 For $s$ in~$ ]-1/3, 1[$  and any positive integer $N,$ which we shall choose hereafter, we write that
\beno
\|\om\|_{B^s_{\f95,1}} &=&
\sum_{j\leq N}2^{js}\|\D_j\om\|_{L^{\f95}}+\sum_{j>N}2^{js}\|\D_j\om\|_{L^{\f95}}\\
& \lesssim &
\sum_{j\leq N}2^{j(s+\f13)}\|\D_j\om\|_{L^{\f32}}+\sum_{j>N}2^{j(s-1)}\|\D_j\na\om\|_{L^{\f95}}\\
& \lesssim &
2^{N(s+\f13)}\|\om\|_{L^{\f32}}+2^{N(s-1)}\|\na\om\|_{L^{\f95}}.
\eeno
 Choosing $\ds N=\biggl[\log_2\biggl(e+ \Bigl(\f{\|\na
\om\|_{L^{\f95}}}{\|\om\|_{L^{\f32}}}\Bigr)^{\f34}\biggr)\biggr]$
yields
$$
\|\om\|_{B^s_{\f95,1}}\lesssim
\|\om\|_{L^{\f32}}^{\f34(1-s)}\|\na\om\|_{L^{\f95}}^{\f34(s+\f13)}.
$$
Using\refeq{BiotSavartomegademoeq2}, we infer that
$$
\|\om\|_{B^s_{\f95,1}}\lesssim
\bigl\|\om_{\frac34}\bigr\|_{L^2}^{1-s}\bigl\|\na\om_{\frac34}\bigr\|_{L^{2}}^{\f13+s}.
$$
Using this inequality  with~$s=\ds \frac 23 -\al$ and
\eqref{BiotSavartBesovanisodemoeq2} gives \beq \label{b.8}
\|\om\|_{\bigl(B^{0}_{2,1}\bigr)_{\rm
h}\bigl(B^{\f12-\al}_{2,1}\bigr)_{\rm v}} \lesssim \bigl\|\,
\om_{\frac34}\bigr\|_{L^2} ^{\frac13+\al} \bigl\|\nabla
\om_{\frac34}\bigr\|_{L^2} ^{1-\al}. \eeq

Now let us prove the following lemma.

\begin{lem}
\label{interpolHtheta}
{\sl  Let us consider~$(\al,\theta)$
in~$]0,1/2[^2$.
Then we have
$$
\|a\|_{\bigl(B^{0}_{2,1}\bigr)_{\rm
h}\bigl(B^{\f12-\al}_{2,1}\bigr)_{\rm v}}\lesssim
\|a\|_{\cH_\theta}^{\al}\|\nabla a\|_{\cH_\theta}^{1-\al}.
$$
}
\end{lem}
\begin{proof}
By definition of~$\|\cdot\|_{\bigl(B^{0}_{2,1}\bigr)_{\rm
h}\bigl(B^{\f12-\al}_{2,1}\bigr)_{\rm v}}$, we have \ben
\|a\|_{\bigl(B^{0}_{2,1}\bigr)_{\rm h}\bigl(B^{\f12-\al}_{2,1}\bigr)_{\rm v}}& =& H_L(a)+V_L(a)\with \nonumber\\
\label{interpolHthetademoeq1} H_L(a) & \eqdefa & \sum_{k\leq \ell}
\|\D_k^{\rm h}\D^{\rm v}_\ell\|_{L^2}
2^{\ell\left (\frac12-\al\right)} \andf\\
 V_L(a) & \eqdefa & \sum_{k> \ell} \|\D_k^{\rm h}\D^{\rm v}_\ell a\|_{L^2}  2^{\ell\left (\frac12-\al\right)} \nonumber
 \een
 In order to estimate~$H_L(a),$ we classically estimate differently  high and low vertical frequencies
  which are here the dominant ones. Using Lemma\refer{lemBern}, we  write that for any~$N$ in~$\ZZ$,
$$
H_L(a)  \lesssim    \sum_{k\leq \ell\leq N} \|\D_k^h\D^v_\ell a\|_{L^2} 2^{\ell\left (\frac12-\al\right)}+ \sum_{\substack{k\leq \ell\\\ell>N}} \|\D_k^h\D^v_\ell \partial_3a\|_{L^2} 2^{-\ell\left (\frac12+\al\right)}.
$$
By definition of the norm of~$\cH_\theta$, we get
$$
H_L(a)  \lesssim  \|a\|_{\cH_\theta}  \sum_{k\leq \ell\leq N} 2^{k\left(\frac12-\theta\right)} 2^{\ell\left (\frac12-\al+\theta\right)}+  \|\partial_3a\|_{\cH_\theta}\sum_{\substack{k\leq \ell\\\ell>N}}
2^{k\left(\frac12-\theta\right)} 2^{-\ell\left (\frac12+\al-\theta\right)}
$$
The hypothesis on~$(\al,\theta)$ imply that
\beno
H_L(a)  & \lesssim &   \|a\|_{\cH_\theta}  \sum_{\ell\leq N}  2^{\ell(1-\al)}+
 \|\partial_3a\|_{\cH_\theta} \sum_{\ell>N}   2^{-\al\ell}\\
& \lesssim &  \|a\|_{\cH_\theta}  2^{N(1-\al)}+
\|\partial_3a\|_{\cH_\theta}  2^{-N\al} \eeno Choosing~$N$ such
that~$\ds
2^N\sim\frac{\|\partial_3a\|_{\cH_\theta}}{\|a\|_{\cH_\theta}}$
gives \beq \label{interpolHthetademoeq2} H_L(a)\lesssim
\|a\|_{\cH_\theta}^\al\|\partial_3a\|_{\cH_\theta}^{1-\al}. \eeq The
term~$V_L(a)$ is estimated along the same lines. In fact, we  get,
by using again Lemma\refer{lemBern}, that \beno V_L(a) & \lesssim &
\sum_{\ell<k\leq N} \|\D_k^h\D^v_\ell a\|_{L^2} 2^{\ell\left
(\frac12-\al\right)} + \sum_{\substack{ \ell<k\\
k>N}}\|\D_k^h\D^v_\ell
\nabla_{\rm h}a\|_{L^2} 2^{\ell\left (\frac12-\al\right)}2^{-k}\\
& \lesssim &  \|a\|_{\cH_\theta} \sum_{\ell <k\leq N}
 2^{\ell\left (\frac12-\al+\th\right)} 2^{k\left(\frac 12-\theta\right)}+
 \|\nabla_{\rm h} a\|_{\cH_\theta}\sum_{\substack{\ell\leq k\\ k>N}} 2^{\ell\left (\frac12-\al+\th\right)}2^{-k\left(\frac12+\theta\right)}\\
 & \lesssim &  \|a\|_{\cH_\theta}  2^{N(1-\al)}+  \|\nabla_{\rm h}a\|_{\cH_\theta}  2^{-N\al}
\eeno Choosing~$N$ such that~$\ds 2^N\sim \frac{\|\nabla_{\rm h}
a\|_{\cH_\theta}}{\|a\|_{\cH_\theta}}$ yields
$$
V_L(a)\lesssim \|a\|_{\cH_\theta}^\al\|\nabla_{\rm h}
a\|_{\cH_\theta}^{1-\al}.
$$
Together with\refeq{interpolHthetademoeq1} and\refeq{interpolHthetademoeq2}, this gives the lemma.
\end{proof}

The application of Lemma  \ref{interpolHtheta} together
with\refeq{b.8} leads to Proposition \ref{BiotSavartBesovaniso}.
\end{proof}

To study product laws between distributions in the anisotropic Besov
spaces, we need to  modify the isotropic para-differential
decomposition of  Bony \cite{Bo81} to the setting of anisotropic
version. We first recall the isotropic para-differential
decomposition from \cite{Bo81}: let $a$ and~Ê$b$ be in~$
\cS'(\R^3)$, \beq \label{pd}\begin{split} &
ab=T(a,b)+\bar{T}(a,b)+ R(a,b)\with\\
& T(a,b)=\sum_{j\in\Z}S_{j-1}a\Delta_jb, \quad
\bar{T}(a,b)=T(b,a), \andf\\
&R(a,b)=\sum_{j\in\Z}\Delta_ja\tilde{\Delta}_{j}b,\quad\hbox{with}\quad
\tilde{\Delta}_{j}b=\sum_{\ell=j-1}^{j+1}\D_\ell a. \end{split} \eeq

As an application of the above basic facts on Littlewood-Paley
theory, we present the following product laws in the anisotropic
Besov spaces.

\begin{lem}
\label{lem2.2qw}
{\sl Let $q\geq 1,$ $p_1\geq p_2\geq 1$ with
$\f1{p_1}+\f1{p_2}\leq 1,$ and $s_1< \f2{p_1},$ $ s_2< \f2{p_2}$
(resp. $s_1\leq \f2{p_1},$ $ s_2\leq \f2{p_2}$ if $q=1$) with
$s_1+s_2>0.$ Let $\sigma_1< \f1{p_1},$ $ \sigma_2< \f1{p_2}$ (resp.
$\sigma_1\leq \f1{p_1},$ $ \sigma_2\leq \f1{p_2}$ if $q=1$) with
$\sigma_1+\sigma_2>0$. Then for $a$ in~$
\bigl(B^{s_1}_{p_1,q}\bigr)_{\rm h}\bigl(B^{\s_1}_{p_1,q}\bigr)_{\rm v}$ and~$b$
in~$\bigl(B^{s_2}_{p_2,q}\bigr)_{\rm h}\bigl(B^{\s_2}_{p_2,q}\bigr)_{\rm v}$ ,
the product~$ab$ belongs to~$
\bigl(B^{s_1+s_2-\f2{p_2}}_{p_1,q}\bigr)_{\rm h}\bigl(B^{\sigma_1+\sigma_2-\f{1}{p_2}}_{p_1,q}\bigr)_{\rm v},$
and \beno \|a
b\|_{\bigl(B^{s_1+s_2-\f2{p_2}}_{p_1,q}\bigr)_{\rm h}\bigl(B^{\sigma_1+\sigma_2-\f{1}{p_2}}_{p_1,q}\bigr)_{\rm v}}\lesssim
\|a\|_{\bigl(B^{s_1}_{p_1,q}\bigr)_{\rm h}\bigl(B^{\s_1}_{p_1,q}\bigr)_{\rm v}}\|b\|_{\bigl(B^{s_2}_{p_2,q}\bigr)_{\rm h}\bigl(B^{\s_2}_{p_2,q}\bigr)_{\rm v}}.
\eeno}
\end{lem}

The proof of the above Lemma is a standard application of Bony's
decomposition \eqref{pd} in both horizontal and vertical variables
and Definition \ref{anibesov}. We skip the details here.

\setcounter{equation}{0}
\section{Proof of the  estimate for the horizontal vorticity}
\label{proofestimateomega}

 The purpose of this section to present the proof of Proposition \ref{inegfondvroticity2D3D}. Let us recall  the first  equation of our reformulation~$(\wt {NS})$ of the incompressible Navier-Stokes equation
 which is
 $$
 \partial_t\om+v\cdot\nabla\om -\D\om=  F\eqdefa \partial_3v^3\om +\partial_2v^3\partial_3v^1-\partial_1v^3\partial_3v^2.
 $$
 As already explained in the second section, we decompose~$F$ as a sum of three terms.  Hence by virtue of \eqref{klips}, we obtain
 \beq
\label{theoxplosaniscalingdemoeq1}
\begin{split}
&\frac 2 3 \bigl\|\om_{\frac 34}(t) \bigr\|_{L^2}^2  + \frac89\int_0^t \bigl\|\nabla \om_{\frac34}(t')\bigr\|_{L^2}^2\,dt' = \frac 2 3 \bigl\||\om_0|^{\frac 3 4}
\bigr\|_{L^2}^2 +\sum_{\ell=1}^3 F_\ell(t) \with\\
&F_1(t) \eqdefa  \int_0^t \int_{\R^3} \p_3v^3 |\om|^{\frac 3 2}\,dx\,dt'\,,\\
&F_2(t) \eqdefa  \int_0^t \int_{\R^3}\bigl(\p_2v^3\p_3v_{\rm curl}^1-\p_1v^3\p_3v_{\rm curl}^2\bigr) \om_{\frac 1 2}\,dx\,dt'\andf\\
&F_3(t) \eqdefa  \int_0^t \int_{\R^3}\bigl(\p_2v^3\p_3v_{\rm
div}^1-\p_1v^3\p_3v_{\rm div}^2 \bigr) \om_{\frac 1 2}\,dx\,dt',
\end{split}
 \eeq
where~$v_{\rm curl}^{\rm h}$ (resp.~$v_{\rm div}^{\rm h}$) corresponds the horizontal
divergence free (resp. curl free) part of the horizontal
vector~$v^{\rm h}=(v^1,v^2),$ which is given by \eqref{a.1},  and
where~$\om_{\frac 1 2}\eqdefa |\om|^{-\frac 12}\,\om$.

Let us start with the easiest term~$F_1$. 
 We first get, by using
integration by parts, that
$$
|F_1(t)|\leq\f32\int_0^t\int_{\R^3}|v^3(t',x)|\,|\p_3\om(t',x)|\,|\om(t',x)|^{\f12}\,dx\,dt'.
$$
Using that
$$
\frac {p-2} {3p} + \frac 2 3 + \frac 2 {3p}=1,
$$
we apply H\"older inequality to get
$$
|F_1(t)|\leq\frac32
\int_0^t\|v^3(t')\|_{L^{\f{3p}{p-2}}}\|\p_3\om(t')\|_{L^{\f32}}\bigl\|\om_{\frac34}(t')\bigr\|_{L^p}^{\f23}\,dt'.
$$
As~$p$ is in~$]4,6[$,  Sobolev embedding  and interpolation
inequality imply that
$$
\bigl\|\om_{\frac34}(t')\bigr\|_{L^p} \lesssim
\bigl\|\om_{\frac34}(t')\bigr\|_{H^{3\left(\frac12-\frac1p\right)}}\lesssim
\bigl\|\om_{\frac34}(t')\bigr\|_{L^2}^{\frac 3p-\frac12}
\bigl\|\nabla \om_{\f34}(t')\bigr\|_{L^2}^{\frac 3 2- \frac 3p}.
$$
Using\refeq{estimbasomega34}, this gives
$$
{
|F_1(t)| \lesssim \int_0^t\|v^3(t')\|_{\dH^{\f12+\f2p}}
\bigl\|\p_3\om_{\frac34}(t')\bigr\|_{L^2}
\bigl\|\om_{\frac34}(t')\bigr\|_{L^2}^{\f13}
}
{
\bigl\|\na\om_{\frac34}(t')\bigr\|_{L^2}^{1-\f2p}\bigl\|\om_{\frac34}(t')\bigr\|_{L^2}^{\f2p-\f13}\,dt'.
}
$$
Applying convex inequality, we obtain
 \ben
|F_1(t)| & \lesssim &
\int_0^t\|v^3(t')\|_{\dH^{\f12+\f2p}}\bigl\|\om_{\frac34}(t')\bigr\|_{L^2}^{\f2p}\bigl\|\na\om_{\frac34}(t')\bigr\|_{L^2}^{2(1-\f1p)} \,dt' \nonumber\\
 \label{a.9}  & \leq &
\f19\int_0^t\bigl\|\na\om_{\frac34}(t')\bigr\|_{L^2}^2\,dt'+C\int_0^t\|v^3(t')\|_{\dH^{\f12+\f2p}}^p\bigl\|\om_{\frac34}(t')\bigr\|_{L^2}^{2}\,dt'.
 \een

\medbreak The other two terms requires a refined way of the
description of the regularity of~$\om_{\frac12}$ and demands a
detailed study  of the anisotropic operator~$\nabla_{\rm h}\D_{\rm
h}^{-1}$ associated with the Biot-Savart's law in horizontal
variables.   Now we  state the lemmas which  allows us to treat  the
terms~$F_2$ and $F_3$ in \eqref{theoxplosaniscalingdemoeq1}.

\begin{lem}
 \label{puisancealphaBesov}
 {\sl Let~$(s,\al)$ be in~$]0,1[^2$ and~$(p,q)$ in~$[1,\infty]^2$.  We consider a function~$G$ from $\R$ to~$\R$ which is H\"olderian
  of exponent~$\al$. Then
 for any~$a$ in the Besov space~$ B^s_{p,q},$ one has
 $$
 \|G(a)\|_{B^{\al s}_{\frac p\al,\frac q\al} }\lesssim \|G\|_{C^\al} \bigl(\|a\|_{B^s_{p,q}}\bigr)^\al\with
\|G\|_{C^\al} \eqdefa \sup_{r\not =r' } \frac
{|G(r)-G(r')|}{|r-r'|^\al} \,\cdotp
 $$
 }
 \end{lem}
 \begin{proof}
 Because the indices $s$ and~$ \al$ are between $0$ and~$1$, we use the definition of Besov spaces
  coming from integral in the physical space (see for instance Theorem 2.36 of \cite{BCD}). Indeed as
$$
\bigl | G(a) -G(b)|\lesssim \|G\|_{C^\al}  |a-b|^\al,
$$
we infer that \beno
\|G(a)-G(a(\cdot+y))\|_{L^{\frac p \al}} & = & \biggl(\int_{\R^d}\bigl |G(a(x))-G(a(x+y))\bigr|^{\frac p\al}  dx \biggr)^{\frac \al p}\\
& \lesssim & \|G\|_{C^\al}\biggl(\int_{\R^d} |a(x)-a(x+y)|^{p}  dx \biggr)^{\frac \al p} \\
& \lesssim & \|G\|_{C^\al}\|a-a(\cdot+y)\|_{L^p}^{\al}. \eeno Then
for any $q<\infty,$ we  write that \beno \biggl(\int_{\R^d} \frac {
\|G(a)-G(a(\cdot+y))\|_{L^{\frac p \al}}^{\frac q\al } } {|y|^{\al
s\times \frac q\al} } \frac {dy}{|y|^d} \biggr)^{\frac \al q}&
\lesssim & \|G\|_{C^\al}\biggl(\int_{\R^d} \frac {
\|a-a(\cdot+y)\|_{L^p} ^q}
{|y|^{sq} } \frac {dy}{|y|^d}\biggr)^{\frac \al q} \\
& \lesssim & \|G\|_{C^\al}\bigl(\|a\|_{B^s_{p,q}}\bigr)^\al. \eeno The case for
$q=\infty$ is identical. This completes the proof of the lemma.
 \end{proof}

 \begin{lem}
 \label{anositropiclemma}
{\sl Let  $\theta$ be in~$]0,1/6[,$ ~$\s$ in~$]3/4,1[$, and $ \ds s=\frac12+1-\frac23\s$. Then we have
 \beq \label{b.1} \begin{split} \Bigl| \int_{\R^3}
\partial_{\rm h}\D_{\rm h}^{-1}f
\partial_{\rm h}a \,\om_{\frac12} dx\Bigr|
& \lesssim  \|f\|_{L^{\frac32}} \|a\|_{\dH^s} \bigl\|\om_{\frac34}\bigr\|_{\dH^\s}^{\frac23}\andf   \\
 \Bigl| \int_{\R^3} \partial_{\rm h}\D_{\rm h}^{-1}f
\partial_{\rm h}a \,\om_{\frac12} dx\Bigr|& \lesssim  \| f\|_{\cH_{\theta}}
\|a\|_{\dH^s} \bigl\|\om_{\frac34}\bigr\|_{\dH^\s}^{\frac23},
\end{split}
\eeq for $\cH_\th$ given by Definition \ref{def2.1ad}.}
 \end{lem}

 \begin{proof}
 Let us observe that~$\ds \om_{\frac12} =G(\om_{\frac34})$ with~$G(r)\eqdefa r|r|^{-\frac13}$. Using Lemma\refer{puisancealphaBesov}, we obtain
 \beq
 \label{anositropiclemmademoeq1}
\bigl\|\om_{\frac12}\bigr\|_{B^{\frac23\s}_{3,3}} \lesssim
\bigl\|\om_{\frac34}\bigr\|_{\dH^\s}^{\frac23}.
 \eeq
 Let us study the product~$\partial_{\rm h}a\om_{\frac12}$. Using Bony's decomposition \eqref{pd} and the Leibnitz formula, we
 write
 \beno
 \partial_{\rm h} a\, \om_{\frac12} & = & T({\partial_{\rm h} a}, \om_{\frac12}) +R(\partial_{\rm h}a,\om_{\frac12})+
 T({\om_{\frac12}}, \partial_{\rm h}a)\\
 & = & \partial_{\rm h}T({\om_{\frac12}}, a) +A(a,\om)\with \\
 A(a,\om) &\eqdefa&
  T({\partial_{\rm h} a}, \om_{\frac12}) +R(\partial_{\rm h}a,\om_{\frac12})
 -T({ \partial_{\rm h}\om_{\frac12}},a).
 \eeno
We first get, by using Lemma \ref{lemBern}, that
 \beno
\|\D_jT({\om_{\frac12}}, a)\|_{L^2}& \lesssim &\sum_{|j-j'|\leq
4}\|S_{j'-1}\om_{\f12}\|_{L^\infty}\|\D_{j'}a\|_{L^2}\\
& \lesssim &
c_{j,2}2^{-j(s+\f23\s-1)}\|\om_{\f12}\|_{B^{\f23\s}_{3,3}}\|a\|_{\dH^s},
\eeno
which together with \eqref{anositropiclemmademoeq1} ensures that
$$
 \bigl\| T({\om_{\frac12}}, a) \|_{\dH^{\frac12} } \lesssim \|a\|_{\dH^s} \bigl\|\om_{\frac34}\bigr\|_{H^\s}^{\f23}.
 $$
Using that the operator~$ \partial_{\rm h}^2\D_{\rm h}^{-1}$ is a bounded Fourier multiplier and  the dual Sobolev embedding~$L^{\frac32}\subset \dH^{-\frac12}$, we get that
\ben
\Bigl| \int_{\R^3} \partial_{\rm h}\D_{\rm h}^{-1}f \partial_{\rm h}
T({\om_{\frac12}}, a)\, dx\Bigr|
& =&
 \Bigl|\int_{\R^3}
\partial^2_{\rm h}\D_{\rm h}^{-1}f T({\om_{\frac12}}, a) \,dx\Bigr|\nonumber \\
&  \leq &
\|f\|_{H^{-\frac 12}}\|T({\om_{\frac12}}, a)\|_{H^{\frac12}} \nonumber \\
  \label{anositropiclemmademoeq2}&  \lesssim &
  \|f\|_{L^{\frac32}} \|a\|_{\dH^s}
  \bigl\|\om_{\frac34}\bigr\|_{\dH^\s}^{\frac23}.
 \een
 In the case of the anisotropic norm, recalling that~$\cH_\theta= H^{-\frac12+\theta,-\theta}$, and using Lemma\refer{isoaniso},  we write
\ben
 \Bigl| \int_{\R^3}
\partial^2_{\rm h}\D_{\rm h}^{-1}f  T({\om_{\frac12}}, a)\, dx\Bigr|  & \leq &
 \|f\|_{\cH_\theta} \|T({\om_{\frac12}}, a)\|_{\dH^{\frac12-\theta,\theta}}\nonumber \\
 & \leq &  \|f\|_{\cH_\theta} \|T({\om_{\frac12}}, a)\|_{\dH^{\frac12}}\nonumber \\
\label{anositropiclemmademoeq2b}& \leq &  \|f\|_{\cH_\theta} \|a\|_{\dH^s}
\bigl\|\om_{\frac34}\bigr\|_{\dH^\s}^{\frac23}.
  \een

 Now let us take into account the anisotropy induced by the operator~$\partial_{\rm h}\D_{\rm h}^{-1}$.
  Hardy-Littlewood-Sobolev inequality implies that~$\partial_{\rm h}\D_{\rm h}^{-1} f$ belongs to~$L^{\frac32}_{\rm v}(L^6_{\rm h})$
   if $f$ is in~$ L^{\f32}.$
  So that it amounts  to prove that~$A(a,\om)$ belongs to~$L^3_{\rm v}(L^{\frac65}_{\rm h}),$ which is simply an anisotropic Sobolev type embedding.
   Because of ~$s<1$, we get, by using Lemma \ref{lemBern}, that
   \beno
   \|\D_jT({\partial_{\rm h}a},\om_{\f12})\|_{L^{\f65}}
   &\lesssim &
   \sum_{|j'-j|\leq  4}\|S_{j'-1}\partial_{\rm h}a\|_{L^2}\|\D_{j'}\om_{\f12}\|_{L^3}\\
  &  \lesssim &\sum_{|j'-j|\leq
   4}c_{j',2}c_{j',3}2^{-j'(s+\f23\s-1)}\|a\|_{\dH^s}\|\om_{\f12}\|_{B^{\f23\s}_{3,3}}\\
&\lesssim &c_{j,\f65}2^{-\f{j}2}\|a\|_{\dH^s}\bigl\|\om_{\frac34}\bigr\|_{\dH^\s}^{\f23}.
\eeno
Along the same line, it is easy to check that the other two
terms in $A(a,\om)$ satisfy the same estimate. This leads to
 \beq\label{b.2}
 \|A(a,\om)\|_{ B^{\f12}_{\frac65,\frac65}} \lesssim  \|a\|_{\dH^s}\bigl\|\om_{\frac34}\bigr\|_{\dH^\s}^{\f23}.
 \eeq
While it follows from Lemma \ref{embeda} that
 $$
B^{\f12}_{\frac65,\frac65} \subset
L^{\frac65}_{\rm h}\bigl((B^{\f12}_{\frac65,\frac65})_{\rm v}\bigr).
 $$
Sobolev type embedding theorem (see for instance Theorem~2.40
of\ccite{BCD}) claims that
 $$
 B^{\f12}_{\frac65,\frac65}(\R) \subset B^0_{3,2}(\R) \subset L^3(\R).
 $$
 As a consequence, by virtue of \eqref{b.2}, we obtain
 \beno
 \Bigl|
\int_{\R^3}
\partial_{\rm h}\D_{\rm h}^{-1}f A(a,\om) \, dx\Bigr| & \lesssim &
\|\partial_{\rm h}\D_{\rm h}^{-1}f\|_{L^{\f32}_v(L^6_{\rm h})}\|A(a,\om)\|_{L^3_{\rm v}(L^{\f65}_h)}\\
&\lesssim&
\|f\|_{L^{\f32}}\|A(a,\om)\|_{L^{\f65}_h(L^3_{\rm v})}\\
& \lesssim &
\|f\|_{L^{\f32}}\|a\|_{\dH^s}\bigl\|\om_{\frac34}\bigr\|_{\dH^\s}^{\f23},
\eeno which together with \refeq{anositropiclemmademoeq2} leads to
the first inequality of \eqref{b.1}.

 In order to prove the second inequality of \eqref{b.1}, we observe that
 $$
\| \nabla_h\D_{\rm h}^{-1} f\|_{\dH^{\frac12+\theta,-\theta} }\lesssim
\|f\|_{\dH^{-\frac12+\theta,-\theta} }=\|f\|_{\cH_\theta}.
 $$
 Thus thanks to \eqref{b.2}, for $\th$ given by the lemma, what we only need to prove now is that
 \beq\label{b.3}
B^{\f12}_{\frac65,\frac65}\subset \dH^{-\frac12-\theta,\theta}. \eeq
 As a matter of fact, using Lemma \ref{isoaniso} and Lemma \ref{lemBern}, we have, for any~$\alpha$ in~$\Bigl]\ds0,\frac12\Bigr[$,
\beno
B^{\f12}_{\frac65,\frac65} & \subset &  \bigl(B^{\f12-\al}_{\frac65,\frac65}\bigr)_{\rm h}\bigl(B^{\al}_{\frac65,\frac65}\bigr)_{\rm v}
\andf\\
 \bigl(B^{\f12-\al}_{\frac65,\frac65}\bigr)_{\rm h}\bigl(B^{\al}_{\frac65,\frac65}\bigr)_{\rm v}
  & \subset &    \bigl(B^{\f12-\al-2\left(\frac56-\frac12\right)}_{2,2}\bigr)_{\rm h}\bigl(B^{\al-\left(\frac56-\frac12\right)}_{2,2})\bigr)_{\rm v}\\
    & \subset &
 \dH^{-\al-\frac16,\al-\frac13}.
\eeno
 Choosing $\ds \al = \frac13+\theta$ gives \eqref{b.3} because
$\theta$ is less than~$\ds\frac16$.  This completes the proof of the
lemma.
 \end{proof}

\medbreak The estimate of~$F_2(t)$ uses the Biot-Savart's law in the
horizontal variables  (namely\refeq{a.1}) and
Lemma\refer{anositropiclemma} with~$f=\partial_3\om$,~$a=v^3$. This
gives for any time~$t<T^\star$ and $\s$ in~$ ]3/4,1[$ that
 $$
 \longformule{
I_\om(t)\eqdefa\Bigl|\int_{\R^3}\bigl(\p_2v^3(t,x)\p_3v_{\rm
curl}^1(t,x)-\p_1v^3(t,x)\p_3v_{\rm curl}^2(t,x) \bigr) \om_{\frac 1
2}(t,x)\,dx\Bigr| }{{}  \lesssim  \|\p_3\om(t) \|_{L^{\f32}} \|
v^3(t)\|_{\dH^{\f32-\f23\s}}
\bigl\|\om_{\frac34}(t)\bigr\|_{\dH^\s}^{\frac23}. }
 $$
By virtue of \eqref{estimbasomega34} and of the interpolation
inequalities between~$L^2$ and~$H^1$, we thus obtain
$$ I_\om(t) \lesssim  \|v^3(t)\|_{H^{\frac12+2\left (\frac12 -\frac\s3\right) }}
\bigl\|\om_{\frac34}(t)\bigr\|_{L^2} ^{2\left (\frac12
-\frac\s3\right)}\bigl\|\nabla\om_{\frac34}(t)\bigr\|_{L^2}^{2\left(\frac12
+\frac\s3\right)}.
 $$
Choosing~$\ds \s = 3\Bigl(\frac 12 -\frac1p\Bigr)$, which is between $3/4$ and~$1$ because~$p$  is between~$4$ and~$6$, gives
$$
I_\om(t)\lesssim \|v^3(t)\|_{H^{\frac12+ \frac 2 p }}
\bigl\|\om_{\frac34}(t)\bigr\|_{L^2} ^{\frac 2p
}\bigl\|\nabla\om_{\frac34}(t)\bigr\|_{L^2}^{2\left(1
-\frac1p\right)} .
$$
Then  by using convexity inequality and time integration, we get
 \beq
 \label{a.10}
 |F_2(t)| \leq \frac19 \int_0^t\bigl\|\nabla \om_{\frac34}(t')\bigr\|_{L^2}^2
 dt' +C \int_0^t  \|v^3(t')\|_{\dH^{\frac12+\frac 2p}}^p\bigl\|\om_{\frac34}(t')\bigr\|_{L^2}
 ^{2}\,dt'.
 \eeq

In order to estimate~$F_3(t)$, we write that
$$
\longformule{ F_3(t) =- \int_0^t
\int_{\R^3}\Bigl(\p_2v^3(t',x)(\p_1\D_{\rm h}^{-1}\p_3^2v^3)(t',x) }
{{} -\p_1v^3(t',x)(\p_2\D_{\rm h}^{-1}\p_3^2v^3)(t',x) \Bigr)\,
\om_{\frac 1 2}(t',x)dxdt'. }$$
As~$\ds \f 2p=1 -\f{2\s}3$, thanks to
interpolation inequality between Sobolev spaces,  we get, by
applying Lemma\refer{anositropiclemma} with~$f=\partial_3^2v^3$ and
~$a=v^3,$
 that
\beno
 |F_3(t)| & \lesssim & \int_0^t \|\p^2_3v^3(t')
\|_{\cH_\th}\| v^3(t')\|_{\dH^{\f32-\f23\s}}
\bigl\|\om_{\frac34}(t')\bigr\|_{\dH^\s}^{\frac23}\,dt'\\
& \lesssim & \int_0^t\|\p^2_3v^3(t') \|_{\cH_\th}\|
v^3(t')\|_{\dH^{\f12+\f2p}}\bigl\|\om_{\frac34}(t')\bigr\|_{L^2}^{\f23(1-\s)}
\bigl\|\na\om_{\frac34}(t')\bigr\|_{L^2}^{\frac23\s}\,dt'\\
& \lesssim & \int_0^t\|\p^2_3v^3(t') \|_{\cH_\th}
\|v^3(t')\|^{\frac p 6}_{\dH^{\f12+\f2p}}\\
&&\qquad\qquad\qquad{}\times\bigl(\|v^3(t')\|^{p}_{\dH^{\f12+\f2p}}\bigl\|\om_{\frac34}(t')\bigr\|^2_{L^2}\bigr)^{\f
1p-\frac 16}
\bigl\|\na\om_{\frac34}(t')\bigr\|_{L^2}^{2\left(\frac12-\frac1p\right)}\,dt'.
\eeno As we have
$$
\frac 1 2+ \frac 1 6 +\biggl(\frac 1 p-\frac1 6\biggr) +\biggl(\frac 12 -\frac1p\biggr)=1,
$$
 applying H\"older inequality ensures that
$$
\longformule{
 |F_3(t)| \lesssim \Bigl(\int_0^t \|\p^2_3v^3(t')\|^{2}_{\cH_\th}\,dt'\Bigr)^{\frac 1{2}}
  \Bigl(\int_0^t \|v^3(t')\|^{p}_{\dH^{\f12+\f2p}}\,dt' \Bigr)^{\frac1{6}}
} { {}\times  \Bigl(\int_0^t \|v^3(t')\|_{H^{\frac12+\frac
2p}}^p\bigl\|\om_{\frac34}(t')\bigr\|_{L^2}^{2}\,dt'\Bigr)^{\frac1{p}-\f16}
\Bigl(\int_0^t \bigl\|\nabla \om_{\frac 3 4}(t')\bigr\|^2_{L^2}
dt'\Bigr)^{\f12-\f1p}. }
$$
Applying the convexity inequality
$$
\displaylines{ a_1a_2a_3 \leq \frac 1{p_1}  a^{p_1}+\frac 1{p_2}
a^{p_2}+\frac 1{p_3}  a^{p_3}\with\cr a_1= C\Bigl(\int_0^t
\|\p^2_3v^3(t')|^{2}_{\cH_\th}\,dt'\Bigr)^{\frac 1{2}}
  \Bigl(\int_0^t \|v^3(t')\|^{p}_{\dH^{\f12+\f2p}}\,dt' \Bigr)^{\frac1{6}},\cr
  a_2= \Bigl(9^{\frac {3(p-2)}{6-p}}\int_0^t
\|v^3(t')\|_{H^{\frac12+\frac
2p}}^p\bigl\|\om_{\frac34}(t')\bigr\|_{L^2}^{2}\,dt'\Bigr)^{\frac1{p}-\f16},\cr
 a_3  = \Bigl(\frac19\int_0^t \bigl\|\nabla \om_{\frac 3
4}(t')\bigr\|^2_{L^2} dt'\Bigr)^{\f12-\f1p}\andf\cr \frac 1 {p_1}=
\frac 2 3\,\virgp\ \ \frac 1 {p_2} =  \frac 1p-\frac 16\andf \frac 1
{p_3} = \frac 12-\frac1p }
$$
 leads to
 \beq
\label{a.8}
\begin{split}
 |F_3(t)| \leq\,& \frac 1 9  \int_0^t
\bigl\|\nabla \om_{\frac 3 4}(t')\bigr\|^2_{L^2} dt' +C \int_0^t
\|v^3(t')\|_{\dH^{\frac12+\frac
2p}}^p\bigl\|\om_{\frac34}(t')\bigr\|_{L^2}
 ^{2}\,dt' \\
&\qquad\qquad\qquad{}+ C\Bigl(\int_0^t \|
v^3(t')\|^{p}_{\dH^{\f12+\f2p}}\,dt' \Bigr)^{\frac {1}{4}}
\Bigl(\int_0^t \|\p^2_3v^3(t') \|^{2}_{\cH_\th}\,dt'\Bigr)^{\frac
{3}{4}}. \end{split} \eeq

\begin{proof}[Conclusion of the proof to
Proposition\refer{inegfondvroticity2D3D}] Resuming the estimates
\eqref{a.9}, \eqref{a.10} and \eqref{a.8} into
\eqref{theoxplosaniscalingdemoeq1}, we obtain \beno \frac 2 3
\bigl\|\om_{\frac34}(t) \bigr\|_{L^2}^2  + \frac59\int_0^t \|\nabla
\om_{\frac 34}(t')\|_{L^2}^2\,dt' & &
\\
&&
\!\!\!\!\!\!\!\!\!\!\!\!\!\!\!\!\!\!\!\!\!\!\!\!\!\!\!\!\!\!\!\!\!\!\!\!\!\!\!\!\!\!\!\!\!\!\!\!\!\!\!\!\!\!\!\!\!\!\!\!\!\!\!\!\!\!\!\!\!\!\!\!\!\!\!\!\!\!\!\!\!\!\!\!\!\!\!\!\!
\leq \frac 2 3 \bigl\||\om_0|^{\frac 3 4} \bigr\|_{L^2}^2{}+
C\Bigl(\int_0^t \| v^3(t')\|^{p}_{\dH^{\f12+\f2p}}\,dt'
\Bigr)^{\frac {1}{4}} \Bigl(\int_0^t \|\p^2_3v^3(t')
\|^{2}_{\cH_\th}\,dt'\Bigr)^{\frac {3}{4}}\\
&&
\quad
{}+C\int_0^t\|
v^3(t')\|_{\dH^{\f12+\f2p}}^{p}
\bigl\|\om_{\frac34}(t')\bigr\|_{L^2}^{2}\,dt'.
\eeno
 Inequality~\eqref{a.10qp} follows from Gronwall lemma once
notice that ~$x^{\frac 14} e^{Cx}\lesssim e^{C'x}$ for~$C'>C$.
\end{proof}

\setcounter{equation}{0}
\section{Proof of the  estimate for the second vertical derivatives of $v^3$}
\label{Sectionestimatedivh}

In this section, we shall present the proof of Proposition
\ref{estimadivhaniso}. Let $\cH_\th$ be given by
Definition\refer{def2.1ad}.  We  get, by taking the $\cH_\th$ inner
product of the $\p_3v^3$ equation of $(\wt{NS})$ with $\p_3v^3,$
that \beq \label{b.4pu}
\begin{split}
\f12\f{d}{dt}\|\p_3v^3(t)\|_{\cH_\th}^2+&\|\na\p_3v^3(t)\|_{\cH_\th}^2 =-\sum_{n=1}^3 \bigl(Q_n(v,v) \, |\, \p_3v^3\bigr)_{\cH_\th}\with\\
 Q_1(v,v) & \eqdefa \bigl(\Id+\partial_3^2\D^{-1} \bigr)(\partial_3v^3)^2 +\partial_3^2\D^{-1}
  \biggl(\sum_{\ell,m=1}^2\partial_\ell v^m\partial_m  v^\ell\biggr) \,,\\
 Q_2(v,v) & \eqdefa  \bigl(\Id+2\partial_3^2\D^{-1} \bigr)\biggl(\sum_{\ell=1}^2 \partial_3 v^\ell
 \partial_\ell v^3\biggr)\andf\\
 Q_3(v,v) & \eqdefa v\cdot \nabla\partial_3 v^3.
 \end{split}
 \eeq
 The  estimate involving~$Q_1$ relies on the
 the following lemma.
 \begin{lem}
\label{estimdivhlemme1}
{\sl Let~$A$ be a bounded Fourier multiplier. If~$p$  and~$\theta$ satisfy
 \beq
\label{Conditionenergyaniso} 0<\th<\f12-\f1p\,\virgp
 \eeq
 then we have
 $$
 \bigl|\bigl(A(D) (fg) \,|\,\partial_3v^3\bigr)_{\cH_\theta}\bigr| \lesssim \|f\|_{H^{\theta,\frac12-\theta-\frac1p} }
\|g\|_{H^{\theta,\frac12-\theta-\frac1p} } \|v^3\|_{H^{\frac12+\frac 2p}}.
 $$
 }
 \end{lem}
 \begin{proof}
 Let us first observe that, for any couple~$(\al,\beta)$ in~$\R^2$, we have, thanks to
Cauchy-Schwartz inequality, that, for any real valued function~$a$
and~$b$, \ben \bigl|(a|b)_{\cH_\theta}\bigr|  &  =  &
\Bigl|\int_{\R^3} |\xi_{\rm h}|^{-1+2\theta-\alpha}
 |\xi_3|^{-2\theta-\beta} \wh a(\xi)|\xi_{\rm h}|^{\alpha}  \,|\xi_3|^{\beta} \wh b(-\xi)d\xi\Bigr|\nonumber\\
\label{estimdivhlemme1demoeq1}
 & \leq & \|a\|_{H^{-1+2\theta-\alpha, -2\theta-\beta}} \|b\|_{H^{\alpha, \beta}} .
\een  As~$A(D)$ is  a bounded Fourier multiplier, applying
\eqref{estimdivhlemme1demoeq1} with~$\al= 0$ and~$\ds\beta= -\frac12
+\frac2p$, we obtain \beq \label{estimdivhlemme1demoeq2}
\bigl|\bigl(A(D) (fg) \,|\,\partial_3v^3\bigr)_{\cH_\theta}\bigr|
\lesssim \|fg\|_{H^{-1+2\theta,\frac12-\frac2p-2\theta}}
\|\partial_3v^3\|_{H^{0,-\frac12+\frac 2p}} \eeq
 Because~$H^{s,s'} = \bigl(B^s_{2,2}\bigr)_{\rm h} \bigl(B^{s'}_{2,2}\bigr)_{\rm v }$
  and thanks to Condition\refeq{Conditionenergyaniso}, law of products of Lemma\refer{lem2.2qw} implies in particular  that
 $$
  \|fg\|_{H^{-1+2\theta,\frac12-\frac2p-2\theta}}  \lesssim \|f\|_{H^{\theta, \frac12-\theta-\frac 1p}}
  \|g\|_{H^{\theta, \frac12-\theta-\frac 1p}}.
 $$
 As~$\|\partial_3v^3\|_{H^{0,-\frac12+\frac 2p}} \lesssim \|v^3\|_{H^{0,\frac12+\frac 2p}}\leq \|v^3\|_{H^{\frac12+\frac 2p}}$,  the lemma is proved.
 \end{proof}

 \medbreak
  Because both~$\p_3^2\D^{-1}$ and $\partial_{\rm h}^2\D_{\rm h}^{-1}$ are bounded Fourier
  multipliers,
 applying Lemma \ref{estimdivhlemme1} with~$f$ and~$g$ of the form~$\partial_{\rm h} v^{\rm h}_{\rm curl}$ or~$\partial_{\rm h} v^{\rm h}_{\rm div}$
  or with~$f=g=\partial_3v^3$ gives,
 $$
\bigl| \bigl(Q_1(v,v) \, |\, \p_3v^3\bigr)_{\cH_\th}\bigr| \lesssim
\|v^3\|_{H^{\frac12+\frac2p}} \bigl(
 \|\om\|_{H^{\theta, \frac12-\theta-\frac 1p}}^2 +\|\partial_3v^3\|_{H^{\theta, \frac12-\theta-\frac 1p}}^2\bigr).
 $$
Because of Condition\refeq{Conditionenergyaniso}, we get, by using
Lemma \ref{isoaniso} and Lemma\refer{BiotSavartomega}, that
$$
\|\om\|_{H^{\theta, \frac12-\theta-\frac 1p}}\leq
\|\om\|_{H^{\frac12-\frac 1p}}\lesssim \bigl\|\,
\om_{\frac34}\bigr\|_{L^2}^{ \frac {p+3}{ 3p} } \bigl\|\nabla
\om_{\frac34}\bigr\|_{L^2} ^{1-\frac1p}.
$$
While it follows from Definition \ref{def2.1ad} that \beno
\|a\|_{H^{\theta, \frac12-\theta-\frac 1p}}^2&=&\int_{\R^3}|\xi_{\rm
h}|^{2\th}|\xi_3|^{1-2\th-\frac2p}|\widehat{a}(\xi)|^2\,d\xi\\
&\leq&\int_{\R^3}|\widehat{a}(\xi)|^{\frac2p}\bigl(|\xi||\widehat{a}(\xi)|\bigr)^{2\bigl(1-\frac1p\bigr)}
|\xi_{\rm h}|^{2(-\f12+\th)}|\xi_3|^{-2\th}\,d\xi. \eeno Applying
H\"older's inequality with measure $|\xi_{\rm
h}|^{2(-\f12+\th)}|\xi_3|^{-2\th}\,d\xi$ yields
$$
\|a\|_{H^{\theta, \frac12-\theta-\frac 1p}} \leq \|a\|_{\cH_\theta}^{\frac 1p} \|\nabla a\|_{\cH_\theta}^{1-\frac 1p},
$$
We then infer that
$$
 \bigl|\bigl(Q_1(v,v) \, |\, \p_3v^3\bigr)_{\cH_\th}\bigr| \lesssim \|v^3\|_{H^{\frac12+\frac2p}} \Bigl(
\bigl\|\, \om_{\frac34}\bigr\|_{L^2}^{ \frac {2(p+3)}{ 3p} }
\bigl\|\nabla \om_{\frac34}\bigr\|_{L^2} ^{2-\frac2p}
+\|\partial_3v^3\|_{\cH_\theta}^{\frac 2p} \|\nabla
\partial_3v^3\|_{\cH_\theta}^{2-\frac 2p}\Bigr).
$$
Convexity inequality ensures \beq \label{estimatedivhdemoeq1}
\begin{split}
 \bigl|\bigl(Q_1(v,v) \, |\, \p_3v^3\bigr)_{\cH_\th}\bigr| &\leq \frac 1 {6} \|\nabla \partial_3v^3\|_{\cH_\theta}^2
 + C \|v^3\|_{H^{\frac12+\frac2p}}^p \|\partial_3v^3\|_{\cH_\theta}^2 \\
&\qquad\qquad\qquad{} + C\|v^3\|_{H^{\frac12+\frac2p}} \bigl\|\,
\om_{\frac34}\bigr\|_{L^2}^{ \frac {2(p+3)}{ 3p} } \bigl\|\nabla
\om_{\frac34}\bigr\|_{L^2} ^{2-\frac2p}.
\end{split}
\eeq

\medbreak The estimates of the two terms involving~$Q_2(v,v)$
and~$Q_3(v,v)$ rely on the following lemma.
\begin{lem}
\label{estimdivhlemme2} {\sl  Let~$A$ be a bounded Fourier
multiplier.  If ~$p$  and~$\theta$ satisfy
Condition\refeq{Conditionenergyaniso} and $\th<\f2p.$ We have,
for~$\ell$ in~$\{1,2\}$,
 $$
\longformule{
 \bigl|\bigl(A(D) (v^\ell\partial_\ell \partial_3v^3)\,|\,\partial_3v^3\bigr)_{\cH_\theta}\bigr|\lesssim
\|v^3\|_{H^{\frac12+\frac2p}} } { {}\times \Bigl( \bigl\|\,
\om_{\frac34}\bigr\|_{L^2}^{ \frac 1 3 +\frac2 p } \bigl\|\nabla
\om_{\frac34}\bigr\|_{L^2} ^{1-\frac2p}
+\|\partial_3v^3\|_{\cH_\theta}^{\frac 2p} \|\nabla
\partial_3v^3\|_{\cH_\theta}^{1-\frac 2p}\Bigr)
\|\nabla\partial_3v^3\|_{\cH_\theta}. }
$$
}
\end{lem}

\begin{proof}
Using\refeq{estimdivhlemme1demoeq2} and the law of product of
Lemma\refer{lem2.2qw} gives, \beno \bigl|\bigl(A(D)
(v^\ell\partial_\ell
\partial_3v^3)\,|\,\partial_3v^3\bigr)_{\cH_\theta}\bigr|&\leq&\|v^\ell\partial_\ell\partial_3v^3\|_{H^{-1+2\th,\f12-\frac2p-2\th}}\|\partial_3v^3\|_{H^{0,-\frac12+\frac2p}}\\
& \lesssim & \|v^\ell\|_{\bigl(B^{1}_{2,1}\bigr)_{\rm
h}\bigl(B^{\f12-\f2p}_{2,1}\bigr)_{\rm v}}
\|\partial_\ell\partial_3v^3\|_{H^{-1+2\theta,\frac12-2\theta}}\|\partial_3v^3\|_{H^{0,-\frac12+\frac2p}}\\
&\lesssim & \|v^\ell\|_{\bigl(B^{1}_{2,1}\bigr)_{\rm
h}\bigl(B^{\f12-\f2p}_{2,1}\bigr)_{\rm v}}
\|\partial_3v^3\|_{H^{2\theta,\frac12-2\theta}}\|v^3\|_{H^{\frac12+\frac2p}}.
\eeno However, notice from Definition \ref{def2.1ad} that \beno
\|\partial_3v^3\|_{H^{2\theta,\frac12-2\theta}}^2&=&\int_{\R^3}|\xi_{\rm
h}|^{4\th}|\xi_3|^{1-4\th}|\widehat{\p_3v^3}(\xi)|^2\,d\xi\\
&\leq& \int_{\R^3}|\xi_{\rm
h}|^{-1+2\th}|\xi_3|^{-2\th}|\xi|^2|\widehat{\p_3v^3}(\xi)|^2\,d\xi=\|\na\p_3v^3\|_{\cH_\th}^2.
\eeno We thus obtain \beno \bigl|\bigl(A(D) (v^\ell\partial_\ell
\partial_3v^3)\,|\,\partial_3v^3\bigr)_{\cH_\theta}\bigr|
&\lesssim & \|v^\ell\|_{\bigl(B^{1}_{2,1}\bigr)_{\rm
h}\bigl(B^{\f12-\f2p}_{2,1}\bigr)_{\rm v}}
\|\nabla\partial_3v^3\|_{\cH_\th}\|v^3\|_{H^{\frac12+\frac2p}}.
\eeno Then Proposition\refer{BiotSavartBesovaniso} leads to the
result.
\end{proof}

\medbreak In order to
estimate~$\bigl(Q_2(v,v)\,|\,\partial_3v^3\bigr)_{\cH_\theta}$, we
write that
\ben
\bigl( (\Id +2\p_3^2\D^{-1}) ( \partial_3v^\ell\partial_\ell v^3)\,|\,\partial_3v^3\bigr)_{\cH_\theta}& = & \cA_1(v^\ell,v^3)+\cA_2(v^\ell,v^3)\with\nonumber\\
\label{estimatedivhdemoeq2-1}
\cA_1(v^\ell,v^3) & \eqdefa &  -\bigl( (\Id +2\p_3^2\D^{-1}) (v^\ell\partial_\ell v^3)\,|\,\partial^2_3v^3\bigr)_{\cH_\theta}\andf\\
\cA_2(v^\ell,v^3) & \eqdefa &  -\bigl( (\Id +2\p_3^2\D^{-1})
(v^\ell\partial_\ell
\partial_3v^3)\,|\,\partial_3v^3\bigr)_{\cH_\theta}.\nonumber \een
Law of product of Lemma\refer{lem2.2qw} implies that \beno
\bigl|\cA_1(v^\ell,v^3)\bigr| &  \lesssim & \|v^\ell\partial_\ell v^3\|_{\cH_\theta} \|\partial_3^2v^3\|_{\cH_\theta}\\
& \lesssim & \|v^\ell\|_{\bigl(B^{1}_{2,1}\bigr)_{\rm
h}\bigl(B^{\f12-\f2p}_{2,1}\bigr)_{\rm v}} \|\partial_\ell
v^3\|_{H^{-\frac 12+\theta,\frac2 p-\theta}}
\|\partial_3^2v^3\|_{\cH_\theta}. \eeno As we have~$\|\partial_\ell
v^3\|_{H^{-\frac 12+\th,\frac2 p-\theta}} \lesssim \|
v^3\|_{H^{\frac 12+\theta,\frac2 p-\theta}}\leq \|v^3\|_{H^{\frac
12+\frac2 p}}$, we infer that
$$
\bigl|\cA_1(v^\ell,v^3)\bigr| \lesssim
\|v^\ell\|_{\bigl(B^{1}_{2,1}\bigr)_{\rm
h}\bigl(B^{\f12-\f2p}_{2,1}\bigr)_{\rm v}} \|v^3\|_{H^{\frac
12+\frac2 p}} \|\partial_3^2v^3\|_{\cH_\theta}.
$$
Because of\refeq{estimatedivhdemoeq2-1},
Proposition\refer{BiotSavartBesovaniso}  and
Lemma\refer{estimdivhlemme2} ensures that
\[
\begin{split}
\bigl|\bigl(Q_2(v,v)\,|\,\partial_3v^3\bigr)_{\cH_\theta}\bigr|
&\lesssim
 \|v^3\|_{H^{\frac12+\frac2p}}\\
 &\!\!\!{}\times \Bigl(
\bigl\|\, \om_{\frac34}\bigr\|_{L^2}^{ \frac 13 +\frac2p }
\bigl\|\nabla \om_{\frac34}\bigr\|_{L^2} ^{1-\frac2p}
+\|\partial_3v^3\|_{\cH_\theta}^{\frac 2p} \|\nabla
\partial_3v^3\|_{\cH_\theta}^{1-\frac 2p}\Bigr)
 \|\nabla \partial_3v^3\|_{\cH_\theta}.
\end{split}
\] Applying convexity inequality yields
\beq
\label{estimatedivhdemoeq2}
\begin{split}
\bigl|\bigl(Q_2(v,v)\,|\,\partial_3v^3\bigr)_{\cH_\theta}\bigr|\,\leq\,
&\frac 1 {6} \|\nabla \partial_3v^3\|_{\cH_\theta}^2
 + C \|v^3\|_{H^{\frac12+\frac2p}}^p \|\partial_3v^3\|_{\cH_\theta}^2 \\
&\qquad\qquad\qquad{} + C\|v^3\|_{H^{\frac12+\frac2p}}^2 \bigl\|\,
\om_{\frac34}\bigr\|_{L^2}^{ \frac {2(p+6)}{ 3p} } \bigl\|\nabla
\om_{\frac34}\bigr\|_{L^2} ^{2\bigl(1-\frac2p\bigr)}.
\end{split}
\eeq

{} \medbreak Finally let us estimate~$\bigl(Q_3(v,v) \, |\,
\p_3v^3\bigr)_{\cH_\th}$.  Lemma\refer{estimdivhlemme2} implies that
\beq
 \label{estimatedivhdemoeq3}
\begin{split}
&\bigl|\bigl( v^{\rm h}\cdot \nabla_{\rm h} \partial_3
v^3\,|\,\partial_3v^3\bigr)_{\cH_\theta}\bigr| \lesssim
\|v^3\|_{H^{\frac12+\frac2p}}\\
&\qquad\quad\qquad {}\times \Bigl( \bigl\|\,
\om_{\frac34}\bigr\|_{L^2}^{ \frac 1 3 +\frac2 p } \bigl\|\nabla
\om_{\frac34}\bigr\|_{L^2} ^{1-\frac2p}
+\|\partial_3v^3\|_{\cH_\theta}^{\frac 2p} \|\nabla
\partial_3v^3\|_{\cH_\theta}^{1-\frac 2p}\Bigr)
\|\nabla\partial_3v^3\|_{\cH_\theta}.
\end{split}
\eeq To estimate~$\bigl( v^{3} \partial_3^2
v^3\,|\,\partial_3v^3\bigr)_{\cH_\theta}$, we write, according to
\eqref{estimdivhlemme1demoeq1}, that
$$
\bigl|(f\,|\,g)_{\cH_\theta}\bigr| \leq
\|f\|_{H^{\theta+\f2p-\f32,-\theta}}\|g\|_{H^{\f12+\theta-\f2p,-\theta}}.
$$ As $\th>\f12-\f2p,$
we get, by applying law of product  of Lemma\refer{lem2.2qw}  and then
Lemma\refer{isoaniso}, that \beno \bigl|\bigl( v^{3} \partial_3^2
v^3\,|\,\partial_3v^3\bigr)_{\cH_\theta}\bigr|
& \leq & \|v^3\partial_3^2v^3\|_{\dH^{\theta+\f2p-\f32,-\theta}}\|\partial_3v^3\|_{\dH^{\f12+\theta-\f2p,-\theta}}\\
& \lesssim &
 \|v^3\|_{\bigl(\dH^{\frac2p}\bigr)_{\rm h}\bigl(B^{\frac12}_{2,1}\bigr)_{\rm v}}
 \| \partial^2_3v^3\|_{\cH_\theta}\|\partial_3v^3\|_{\dH^{\f12+\theta-\f2p,-\theta}}  \\
 & \lesssim & \|v^3\|_{\dH^{\f12+\f2p}}\|
 \partial^2_3v^3\|_{\cH_\theta}\|\partial_3v^3\|_{\dH^{\f12+\theta-\f2p,-\theta}},
 \eeno
This along with the interpolation inequality which claims that
 \beno
\|\partial_3v^3\|_{\dH^{\f12+\theta-\f2p,-\theta}}\leq
\|\p_3v^3\|_{\cH_\th}^{\frac2p}\|\na_h\p_3v^3\|_{\cH_\th}^{1-\f2p},
\eeno ensures
 \beno  \bigl|\bigl( v^{3} \partial_3^2
v^3\,|\,\partial_3v^3\bigr)_{\cH_\theta}\bigr|\lesssim
\|v^3\|_{\dH^{\f12+\f2p}}\|\p_3v^3\|_{\cH_\th}^{\frac2p}\|
 \na\partial_3v^3\|_{\cH_\theta}^{2-\frac2p}.\eeno
Due to \eqref{estimatedivhdemoeq3} and convexity inequality, we thus
obtain \beq \label{estimatedivhdemoeq4p}
\begin{split}
\bigl|\bigl(Q_3(v,v)\,|\,\partial_3v^3\bigr)_{\cH_\theta}\bigr|\leq
\,&\,\frac 1 {6} \|\nabla \partial_3v^3\|_{\cH_\theta}^2
 + C \|v^3\|_{H^{\frac12+\frac2p}}^p \|\partial_3v^3\|_{\cH_\theta}^2 \\
&\qquad\qquad\qquad{} + C\|v^3\|_{H^{\frac12+\frac2p}}^2 \bigl\|\,
\om_{\frac34}\bigr\|_{L^2}^{ \frac {2(p+6)}{ 3p} } \bigl\|\nabla
\om_{\frac34}\bigr\|_{L^2} ^{2\bigl(1-\frac2p\bigr)}.
\end{split}
\eeq

\bigbreak


Now we are in a position to complete the proof of Proposition
\ref{estimadivhaniso}.

\medbreak

\begin{proof}[Conclusion of the proof to Proposition
\ref{estimadivhaniso}] By resuming the estimates
\eqref{estimatedivhdemoeq1}, \eqref{estimatedivhdemoeq2} and
\eqref{estimatedivhdemoeq4p} into\refeq{b.4pu}, we obtain
\beq\label{b.26}
\begin{split}
\f{d}{dt}\|\p_3v^3(t)\|_{\cH_\th}^2+&\|\na\p_3v^3(t)\|_{\cH_\th}^2\\
&{}\leq{}
C\Bigl(\|v^3\|_{\dH^{\f12+\f2p}}^2\bigl\|\om_{\frac34}\bigr\|_{L^2}^{2\bigl(\f13+\f2p\bigr)}\bigl\|\na\om_{\frac34}\bigr\|_{L^2}^{2\bigl(1-\f2p\bigr)}
\\
&{}+\|v^3\|_{\dH^{\f12+\f2p}}^p\|\p_3v^3\|_{\cH_\th}^2+\|v^3\|_{\dH^{\f12+\f2p}}\bigl\|\om_{\frac34}\bigr\|_{L^2}^{2\bigl(\f13+\f1p\bigr)}\bigl\|\na\om_{\frac34}\bigr\|_{L^2}^{2\bigl(1-\f1p\bigr)}
\Bigr).
\end{split}
\eeq
On the other hand, Inequality\refeq{initialdataHtheta}  claims that
$\|\p_3v_0^3\|_{\cH^{\th}}\lesssim\|v_0\|_{\dH^{\f12}} $. Thus
Gronwall's inequality allows to conclude  the proof of  Proposition
\ref{estimadivhaniso}.\end{proof}

\setcounter{equation}{0}
\section{The closure of the estimates to horizontal vorticity and divergence}
\label{Gronwall+}

The  main step of the proof of Proposition\refer{SophisticatedGronwall} is the proof of  the following estimate, for any~$t$ in~$[0,T^\star[$.
 \beq
 \label{b.28}
\bigl\|\om_{\f34}(t)\bigr\|_{L^2}^{2\bigl(\f{p+3}3\bigr)}+\bigl\|\na\om_{\frac34}\bigr\|_{L^2_t(L^2)}^{2\bigl(\f{p+3}3\bigr)}
\leq C \|\Om_0\|_{L^{\f32}}^{\f{p+3}2} \exp\biggl(C\exp
\Bigl(C\int_0^t\|v^3(t)\|_{\dH^{\f12+\f2p}}^p\,dt'\Bigr)\biggr).
\eeq In order to do it, let us  introduce the notation
\beq\label{b.289q} \ds e(T)\eqdefa C\exp \Bigl(C\int_0^T
\|v^3(t)\|_{\dH^{\frac 12+\frac2 p}}^p dt\Bigr). \eeq where the
constant~$C$ may change from line to line. As~$(a+b)^{\frac 34} \sim
a^{\frac34}+b^{\frac34}$, Proposition\refer{estimadivhaniso} implies
that \beq
 \label{SophisticatedGronwalldemoeq2}
\begin{split}
\Bigl(\int_0^t \|\p^2_3v^3(t')
\|^{2}_{\cH_\th}\,dt'\Bigr)^{\frac {3}{4}}e(T) \lesssim e(T)\bigl(\|\Om_0\|_{L^{\f32}}^{\f32}+V_1(t)+V_2(t)\bigr)\with\qquad\qquad\\
V_1(t) \eqdefa
 \biggl(\int_0^t\|v^3(t')\|_{H^{\frac12+\frac2p}}
\bigl\| \, \om_{\frac34}(t')\bigr\|_{L^2}^{2\left(\frac13+\frac 1
p\right)}
\bigl\|\na\om_{\frac34}(t')\bigr\|_{L^2}^{2\bigl(1-\f1p\bigr)}dt'
\biggr)^{\frac34}
\andf \\
V_2(t) \eqdefa \biggl(\int_0^t\|v^3(t')\|_{H^{\frac12+\frac2p}}^2
\bigl\| \,
\om_{\frac34}(t')\bigr\|_{L^2}^{2\left(\frac13+\frac2p\right)}
\bigl\|\na\om_{\frac34}(t')\bigr\|_{L^2}^{2\bigl(1-\f2p\bigr)}dt'
\biggr)^{\frac34}.\qquad
\end{split}
\eeq Let us estimate the two terms~$V_j(t), j=1,2$. Applying
H\"older inequality gives
 \beno
  V_1(t)  &\leq &
\biggl(\int_0^t\|v^3(t')\|_{\dH^{\frac12+\frac2p}}^p
\bigl\| \om_{\frac34}(t')\bigr\|_{L^2}^{2\left(\frac13+\frac1p\right)p}\,dt'\biggr)^{\frac34
\times\frac1 p}
\biggl(\int_0^t\bigl\|\na\om_{\frac34}(t')\bigr\|_{L^2}^{2}dt' \biggr)^{\frac34\left(1-\frac1p\right)}\\
  & \leq &   \biggl(\int_0^t\|v^3(t')\|_{\dH^{\frac12+\frac2p}}^p
\bigl\| \, \om_{\frac34}(t')\bigr\|_{L^2}^{2\bigl(\frac
{p+3}3\bigr)}\,dt'\biggr)^{\frac3{4 p}}
\biggl(\int_0^t\bigl\|\na\om_{\frac34}(t')\bigr\|_{L^2}^{2}dt'
\biggr)^{\frac34\left(1-\frac1p\right)}.
 \eeno
 As we have
$$
1-\frac34\Bigl(1-\frac1p\Bigr) = \frac {p+3}{4p}\,\virgp
$$
convexity inequality implies that, for any~$t$ in~$[0,T]$, \beq
\label{SophisticatedGronwalldemoeq3} e(T) V_1(t) \leq \frac19
\int_0^t\bigl\|\na\om_{\frac34}(t')\bigr\|_{L^2}^{2}dt' + e(T)
\biggl(\int_0^t\|v^3(t')\|_{H^{\frac12+\frac2p}}^p \bigl\|
\om_{\frac34}(t')\bigr\|_{L^2}^{2\bigl(\frac
{p+3}3\bigr)}dt'\biggr)^{\frac 3{p+3}}. \eeq Now let us estimate the
term~$V_2(t)$. Applying H\"older inequality yields \beno V_2(t)  &
\leq & \biggl(\int_0^t\|v^3(t')\|_{H^{\frac12+\frac2p}}^p \bigl\|
\om_{\frac34}(t')\bigr\|_{L^2}^{2\left(\frac13+\frac2p\right)\frac
p2}dt'\biggr)^{\frac34\times\frac2 p}
\biggl(\int_0^t\bigl\|\na\om_{\frac34}(t')\bigr\|_{L^2}^{2}dt' \biggr)^{\frac34\left(1-\frac2p\right)}\\
&  \leq &   \biggl(\int_0^t\|v^3(t')\|_{H^{\frac12+\frac2p}}^p
\bigl\|  \om_{\frac34}(t')\bigr\|_{L^2}^{2\bigl(\frac {p
+6}6\bigr)}dt'\biggr)^{\frac3{2 p}}
\biggl(\int_0^t\bigl\|\na\om_{\frac34}(t')\bigr\|_{L^2}^{2}dt'
\biggr)^{\frac34\left(1-\frac2p\right)}.
 \eeno
 As we have
$$
1-\frac34\Bigl(1-\frac2p\Bigr) = \frac {p+6}{4p}\,\virgp
$$
convexity inequality implies that
\beq
\label{SophisticatedGronwalldemoeq4} e(T) V_2(t)
 \leq \frac19
\int_0^t\bigl\|\na\om_{\frac34}(t')\bigr\|_{L^2}^{2}dt' + e(T)
\biggl(\int_0^t\|v^3(t')\|_{H^{\frac12+\frac2p}}^p
\|  \om_{\frac34}(t')\|_{L^2}^{2\bigl(\frac{p+6}6\bigr)}dt'\biggr)^{\frac 6{p+6}}.
\eeq
 Let us notice that the power of~$\bigl\|  \om_{\frac34}\bigr\|_{L^2}$ here is not the
same as that in Inequality\refeq{SophisticatedGronwalldemoeq3}.
Applying H\"older inequality with
$$
q = \frac {p+3} 3\times\frac 6{p+6} =  2\,\frac {p+3}{p+6}
$$
and with the measure~$\|v^3(t')\|_{H^{\frac12+\frac2p}}^pdt'$ gives
$$
\longformule{ \biggl(\int_0^t\|v^3(t')\|_{H^{\frac12+\frac2p}}^p
\bigl\|  \om_{\frac34}(t')\bigr\|_{L^2}^{2\bigl(\frac
{p+6}6\bigr)}dt'\biggr)^{\frac 6{p+6}} \leq
\biggl(\int_0^t\|v^3(t')\|_{H^{\frac12+\frac2p}}^pdt'\biggr)^{\bigl(1-\frac
1q\bigr)\times\frac6{p+6}} } {{}\times
\biggl(\int_0^t\|v^3(t')\|_{H^{\frac12+\frac2p}}^p \bigl\|
\om_{\frac34}(t')\bigr\|_{L^2}^{2\bigl(\frac
{p+3}3\bigr)}\,dt'\biggr)^{\frac 3{p+3}}. }
$$
By definition of~$e(T)$, we have
$$
\biggl(\int_0^t\|v^3(t')\|_{H^{\frac12+\frac2p}}^pdt'\biggr)^{\bigl(1-\frac
1q\bigr)\times\frac6{p+6}} e(T)\leq e(T).
$$
Thus we deduce from\refeq{SophisticatedGronwalldemoeq4} that
$$
e(T) V_2(t) \leq \frac19
\int_0^t\bigl\|\na\om_{\frac34}(t')\bigr\|_{L^2}^{2}dt' + e(T)
\biggl(\int_0^t\|v^3(t')\|_{H^{\frac12+\frac2p}}^p
\bigl\|\om|_{\frac34}(t')\bigr\|_{L^2}^{2\bigl(\frac
{p+3}3\bigr)}dt'\biggr)^{\frac 3{p+3}}.
$$
Plugging this inequality and\refeq{SophisticatedGronwalldemoeq3} into\refeq{SophisticatedGronwalldemoeq2} gives, for any~$t$ in~$[0,T]$,
$$
\longformule{ \Bigl(\int_0^t \|\p^2_3v^3(t')
\|^{2}_{\cH_\th}\,dt'\Bigr)^{\frac {3}{4}}e(T) \leq \frac 2 9
\int_0^t\bigl\|\na\om_{\frac34}(t')\bigr\|_{L^2}^{2}dt'+e(T)\|\Om_0\|_{L^{\f32}}^{\f32}
} {{} + e(T) \biggl(\int_0^t\|v^3(t')\|_{H^{\frac12+\frac2p}}^p
\bigl\| \om_{\frac34}(t')\bigr\|_{L^2}^{2\bigl(\frac
{p+3}3\bigr)}dt'\biggr)^{\frac 3{p+3}}. }
$$
Hence thanks to Proposition\refer{inegfondvroticity2D3D}, we deduce
that
$$
\longformule{
\frac 23  \|  \om_{\frac34}(t)\|_{L^2}^{2}
 +\frac 13 \int_0^t\|\nabla \om_{\frac34}(t')\|_{L^2}^2\,dt' \leq  \|\Om_0\|_{L^{\f32}}^{\f32}e(T)
} {{} + e(T) \biggl(\int_0^t\|v^3(t')\|_{H^{\frac12+\frac2p}}^p
\|\om_{\frac34}(t')\|_{L^2}^{2\bigl(\frac
{p+3}3\bigr)}dt'\biggr)^{\frac 3{p+3}}. }
$$
Taking the power~$\ds \frac {p+3}3$ of this inequality  and using
that ~$(a+b)^{ \frac {p+3}3} \sim a^{ \frac {p+3}3}+b^{ \frac
{p+3}3}$,  we obtain for any~$t$ in~$[0,T]$,
$$
\longformule{
 \bigl\|  \om_{\frac34}(t)\bigr\|^{2\bigl(\frac {p+3}3\bigr)}_{L^2}
 +\biggl(\int_0^t\bigl\|\nabla \om_{\frac34}(t')\bigr\|_{L^2}^2\,dt'\biggr)^{\frac{p+3}3}
  \leq \| \Om_0\|_{L^{\f32}}^{\frac {p+3}2}e(T)
 }
{{} + e(T)\int_0^t\|v^3(t')\|_{\dH^{\frac12+\frac2p}}^p
\bigl\| \om_{\frac34}(t')\bigr\|_{L^2}^{2\bigl(\frac {p+3}3\bigr)}dt'. }
$$
Then Gronwall lemma leads to Inequality \eqref{b.28}.

On the other hand, it follows from
Proposition\refer{estimadivhaniso}  that, for any $t<T^\ast,$ \beno
&& \|\p_3v^3(t)\|_{\cH_\th}^2
 +\int_0^t\|\na\p_3v^3(t')\|_{\cH_\th}^2\,dt'\\
&&\qquad\qquad \leq
 e(t)\biggl(\|\Om_0\|_{L^{\f32}}^2+\|v^3\|_{L^p_t(\dH^{\f12+\f2p})}\bigl\|\om_{\frac34}\bigr\|_{L^\infty_t(L^2)}^{2\bigl(\frac {p+3} {3p}\bigr)}
 \bigl\|\na\om_{\frac34}\bigr\|_{L^2_t(L^2)}^{2\bigl(1-\f1p\bigr)}\\
 &&\qquad\qquad\qquad\qquad\qquad\qquad\qquad{}+\|v^3\|_{L^p_t(\dH^{\f12+\f2p})}^2\bigl\|\om_{\frac34}\bigr\|_{L^\infty_t(L^2)}^{2\bigl(\frac {p+6} {3p}\bigr)}
 \bigl\|\na\om_{\frac34}\bigr\|_{L^2_t(L^2)}^{2\bigl(1-\f2p\bigr)}\biggr).
 \eeno
Resuming the estimate \eqref{b.28} into the above inequality
concludes the proof of Proposition\refer{SophisticatedGronwall}.

\setcounter{equation}{0}
\section{Proof of the end point blow up theorem}
The proof of Theorem\refer{blowupBesovendpoint} relies on the following  lemma.
\begin{lem}
\label{lemtheoblowupendpoint}
{\sl  Let~$(p_{k,\ell})_{1\leq k,\ell\leq3}$
be a sequence of~$]1,\infty[^9$ and $v=(v^1,v^2,v^3)$ be a smooth
divergence free vector field. Then for the norm $\|\cdot\|_{\cB_p}$
given by Definition \ref{def2.2}, we have
$$
\bigl|(v\cdot\nabla v|v)_{H^\frac12}\bigr| \lesssim
\sum_{k,\ell}\|\partial_\ell v^k\|_{\cB_{p_{k,\ell}}} \|v\|_{
H^{\frac12}}^{\frac2{p_{k,\ell}}} \|\nabla v\|_{
H^{\frac12}}^{2-\frac2{p_{k,\ell}}}.
$$
}
\end{lem}
\begin{proof}
Let us choose on~$H^{\frac12}$ the following inner product
$$
(a|b)_{\dH^{\frac12}} = \sum_{j\in\Z} 2^j(\D_ja|\D_jb)_{L^2}.
$$
 We use Bony's decomposition
\eqref{pd} to deal with the product function $v\cdot\na v.$ Namely,
we write \beq \label{bonydecomp}
\begin{split} v^\ell\p_\ell v^k &=T({v^\ell},\p_\ell v^k)+T({\p_\ell v^k},v^\ell)+R(v^\ell,\p_\ell v^k).
\end{split}
\eeq Let us start with the terms~$T({\partial_\ell v^k}, v^\ell)$.
The support of the Fourier transform of the function~$S_{j'-1} \partial_\ell
v^k\D_{j'}v^\ell$ is included in a ring of the type~$2^{j'}\wt\cC$.
Thus according to Definition \ref{def2.2}, we have \beno
\|\D_jT({\partial_\ell v^k}, v^\ell)\|_{L^2} &\leq & \sum_{|j'-j|\leq 4} \|S_{j'-1}\partial_\ell v^k\D_{j'}v^\ell\|_{L^2}\\
&\leq & \sum_{|j'-j|\leq 4} \|S_{j'-1}\partial_\ell
v^k\|_{L^\infty}\|\D_{j'}v^\ell\|_{L^2}\\ &\lesssim &  \|\p_\ell
v^k\|_{\cB_{p_{k,\ell} }}
2^{j\left(1-\frac1{p_{k,\ell}}\right)}\sum_{|j'-j|\leq
4}2^{j'\left(1-\frac1{p_{k,\ell}}\right)} \|\D_{j'}v^\ell\|_{L^2}.
\eeno Now let us write that
$$
\longformule{ 2^j\bigl|\bigl(\D_jT({\partial_\ell v^k},v^k)\, \big|\,
\D_jv^k\bigr)_{L^2}\bigr|  \lesssim
 \|\p_\ell v^k\|_{\cB_{p_{k,\ell} }}
\bigl(2^{\frac j2}\|\D_jv\|_{L^2}\bigr)^{\frac1{p_{k,\ell}}}
\bigl(2^{\frac {3j} 2}\|\D_jv\|_{L^2}\bigr)^{1-\frac1{p_{k,\ell}}} }
{{}\times \sum_{|j'-j|\leq 4}2^{\frac {j-j'} 2}\bigl(2^{\frac
{j'}2}\|\D_{j'}v\|_{L^2}\bigr)^{\frac1{p_{k,\ell}}} \bigl(2^{\frac
{3j'} 2}\|\D_{j'}v\|_{L^2}\bigr)^{1-\frac1{p_{k,\ell}}} . }
$$
Using  the characterization of Sobolev norms in term of
Littlewood-Paley theory,  we get
\beq
\label{lemtheoblowupendpointdemoeq1}  \sum_{ k,\ell=1}^3
\sum_{j\in\Z} 2^j \bigl|(\D_jT({\partial_\ell v^k},v^\ell) \, | \,
\D_jv^k)_{L^2}\bigr| \lesssim \sum_{1\leq k,\ell\leq
3}\|\partial_\ell v^k\|_{\cB_{p_{k,\ell}}} \|v\|_{
H^{\frac12}}^{\frac2{p_{k,\ell}}} \|\nabla v\|_{
H^{\frac12}}^{2-\frac2{p_{k,\ell}}}.
\eeq

The terms~$R(\partial_\ell v^k,v^\ell)$ are a little bit more
delicate. The support of the Fourier transform of~$\D_{j'}
\partial_\ell v^k\wt{\D}_{j'}v^\ell$ is included in a ball  of the
type~$2^{j'}\wt B$. Thus we have
$$
\D_j R(\partial_\ell v^k,v^\ell) = \sum_{j'\geq j-N_0}
\D_j\bigl(\D_{j'}\partial_\ell v^k\wt{\D}_{j'}v^\ell\bigr).
$$
Because of the divergence free condition of $v$,  we can write
$$
 \sum_{\ell=1}^3 \D_j R(\partial_\ell v^k,v^\ell) = \sum_{\ell=1}^3
 \partial_\ell\!\!\sum_{j'\geq j-N_0}\D_j\bigl(\D_{j'} v^k\wt{\D}_{j'}v^\ell\bigr).
$$
Using the fact that the Fourier transform  of~$\D_{j'}$ is supported
in a ring of the type~$2^{j'}\cC$, we can write that \beq
\label{lemtheoblowupendpointdemoeq11} \D_{j'} v^k=\sum_{\ell'=1}^3
2^{-j'} \wt \D_{j'}^{\ell'} \D_{j'}
\partial_{\ell'}v^k\with\wt\D_{j'}^{\ell'} a\eqdefa \cF^{-1}
\bigl(\phi^{\ell'}(2^{-j'}\xi)\wh{a}\bigr) \eeq
where~$\phi^{\ell'},$ for $\ell'=1,2,3,$ are function
of~$\cD(\R^3\setminus\{0\})$ (see for instance  page 56 of
\ccite{BCD} for the details). We thus obtain
$$
 \Bigl\|\sum_{\ell=1}^3 \D_j R(\partial_\ell v^k,v^\ell) \Bigr\|_{L^2} \lesssim \sum_{\ell=1}^3
 \sum_{j'\geq j-N_0}  2^{-(j'-j)}  2^{2j'\left(1-\frac1{p_{k,\ell}}\right)} \|\partial_{\ell}v^k\|_{\cB_{p_{k,\ell}}}  \|\wt{\D}_{j'}
 v\|_{L^2},
$$
from which, we infer \beno
I_R(v)&\eqdefa&  \sum_{ k,\ell=1}^3 \sum_{j\in\Z} 2^j  \bigl |  \bigl( \D_j R(\partial_\ell v^k,v^\ell)\,\big|\, \D_jv^k\bigr)_{L^2}\bigr|\\
&\lesssim& \sum_{k=1}^3 \sum_{j\in\Z} 2^j \Bigl \| \sum_{\ell=1}^3  \D_j R(\partial_\ell v^k,v^\ell) \Bigr\|_{L^2}\|\D_jv\|_{L^2}\\
&\lesssim &  \sum_{ k,\ell=1}^3 \|\partial_\ell v^k\|_{\cB_{p_{k,\ell}}} \\
&&\qquad\quad{}\times \!\!\! \sum_{\substack{j,j'\in\Z\\ j'\geq j-N_0}}
2^{-(j'-j)\left(\frac12+\frac1{p_{k,\ell}}\right)}\bigl(2^{\frac
{j'}2}\|\wt{\D}_{j'}v\|_{L^2}\bigr)^{\frac1{p_{k,\ell}}}
\bigl(2^{\frac {3j'}
2}\|\wt{\D}_{j'}v\|_{L^2}\bigr)^{1-\frac1{p_{k,\ell}}} \\
&&\qquad\qquad\qquad\qquad\qquad\qquad\quad{}\times \bigl(2^{\frac
{j}2}\|\D_jv\|_{L^2}\bigr)^{\frac1{p_{k,\ell}}} \bigl(2^{\frac {3j}
2}\|\D_jv\|_{L^2}\bigr)^{1-\frac1{p_{k,\ell}}}. \eeno
 Using the convolution law of~$\ZZ,$ we deduce that
\beq \label{lemtheoblowupendpointdemoeq2} I_R(v)
\lesssim
  \sum_{ k,\ell=1}^3\|\partial_\ell v^k\|_{\cB_{p_{k,\ell}}} \|v\|_{H^{\frac12}}^{\frac2{p_{k,\ell}}}
\|\nabla v\|_{H^{\frac12}}^{2-\frac2{p_{k,\ell}}}. \eeq To deal with
the terms of the form~$T({v^\ell},\partial_\ell v^k)$ in
\eqref{bonydecomp}, we  use the skew symmetry property of the
operator~$v\cdot\nabla$. Then we  follow\ccite{chemin10}. As the
support of the Fourier transform of~$\displaystyle S_{j'-1}a
 \D_{j'}b$ is included in a ring of the type~$2^{j'}\wt \cC$, we
write
\begin{eqnarray*}
\D_j\sum_{j'} S_{j'-1}v^\ell  \D_{j'}\partial_{\ell} w  & = & S_{j-1}  v^\ell   \D_{j}\partial_{\ell} w
+\sum_{\ell=1}^2 R_{j,\ell}^k (v,w) \with\\
R_{j,\ell}^1(v^\ell,w) &\eqdefa &  \sum_{|j'-j|\leq 4} \bigl[\D_j,
S_{j'-1}v^\ell\bigr]
 \D_{j'}\partial_{\ell} w\quad\hbox{and}\\
R_{j,\ell}^2(v^\ell,w) &\eqdefa & \sum_{|j'-j|\leq 4}
(S_{j'-1}v^\ell-S_{j-1}v^\ell)\D_j
 \D_{j'}\partial_{\ell} w.
\end{eqnarray*}
By definition of the space~$\cB_p$ in Definition \ref{def2.2},
Lemma~2.97 of\ccite{BCD} implies that \ben
\|R_{j,\ell}^1(v,w)\|_{L^2}  & \lesssim &2^{-j}\sum_{|j'-j|\leq 4}
\|\nabla S_{j'-1} v^\ell\|_{L^\infty} \left\| \D_{j'}\partial_{\ell}w\right\|_{L^2}\nonumber\\
& \lesssim &2^{-j}\sum_{|j'-j|\leq 4}\sum_{\ell'=1}^3 \|S_{j'-1}
\partial_{\ell'} v^\ell\|_{L^\infty} \left\|
\D_{j'}\partial_{\ell}w\right\|_{L^2}
\label{regulNS2DdemoLemme1demoeqA1}\\ & \lesssim & \sum_{\ell'=1}^3
2^{j\left(1-\frac2{p_{\ell,\ell'} } \right)}\|\partial_{\ell'}
v^\ell\|_{\cB_{p_{\ell,\ell'}}}\sum_{|j'-j|\leq 4}
 \left\| \D_{j'}\partial_{\ell}w\right\|_{L^2}. \nonumber
\een In order to estimate~$\|R_{j,\ell}^2(v,w)\|_{L^2}$, we use
Lemma  \ref{lemBern} to get
$$
\|R_{j,\ell}^2(v^\ell,w)\|_{L^2}  \lesssim\!\!\!\!
\sum_{\substack{|j'-j|\leq 4\\j''\in [j-1,j'-1]}}
\|\D_{j''}v^\ell\|_{L^\infty} \|\D_j\D_{j'} \p_\ell w\|_{L^2}.
$$
Notice that \refeq{lemtheoblowupendpointdemoeq11} ensures that
$$
\|\D_{j}v^\ell  \|_{L^\infty} \lesssim
2^{-j}\sum_{\ell'=1}^3\|\D_{j} \partial_{\ell'} v^\ell\|_{L^\infty}.
$$
By virtue of Definition \ref{def2.2}, this implies that
$$
\|\D_j v^\ell\|_{L^\infty} \lesssim \sum_{\ell'=1}^3 2^{j\left(1-\frac 2{p_{\ell,\ell'} }\right )}\|\partial_{\ell'} v^\ell\|_{\cB_{p_{\ell,\ell'}}}.
$$
We thus infer that \beq \label{regulNS2DdemoLemme1demoeqA1a}
\|R_{j,\ell}^2(v,w)\|_{L^2} \lesssim \sum_{\ell'=1}^3
2^{j\left(1-\frac2{p_{\ell,\ell'} } \right)} \|\partial_{\ell'}
v^\ell\|_{\cB_{p_{\ell,\ell'}}}\sum_{|j'-j|\leq 4}
 \left\| \D_{j'}\partial_{\ell}w\right\|_{L^2}.
\eeq Because of the divergence  free  on~$v$, we have
$$
(S_{j-1}  v\cdot   \D_{j} w|\D_jw)_{L^2} =0,
$$
this together with\refeq{regulNS2DdemoLemme1demoeqA1} and
\eqref{regulNS2DdemoLemme1demoeqA1a} gives rise to
\[
 \begin{split}
 \Bigl| \sum_{ k,\ell=1}^3& \sum_{j\in\Z} 2^j
\bigl( \D_j T({v^\ell},\partial_\ell v^k)\ \big|\
\D_jv^k\bigr)_{L^2}\Bigr|\\
\lesssim & \sum_{1\leq k,\ell\leq3} \|\partial_\ell v^k\|_{\cB_{p_{k,\ell}}} \\
&\qquad\quad{}\times \!\!\! \sum_{\substack{ j,j'\in\Z \\|j'-j|\leq 4}}
2^{(j-j')\left(\frac12-\frac1{p_{k,\ell}}\right)}\bigl(2^{\frac
{j'}2}\|\D_{j'}v\|_{L^2}\bigr)^{\frac1{p_{k,\ell}}} \bigl(2^{\frac
{3j'}
2}\|\D_{j'}v\|_{L^2}\bigr)^{1-\frac1{p_{k,\ell}}} \\
&\qquad\qquad\qquad\qquad\qquad\qquad\quad{}\times \bigl(2^{\frac
{j}2}\|\D_jv\|_{L^2}\bigr)^{\frac1{p_{k,\ell}}} \bigl(2^{\frac {3j}
2}\|\D_jv\|_{L^2}\bigr)^{1-\frac1{p_{k,\ell}}}\\
\lesssim &
  \sum_{ k,\ell=1}^3\|\partial_\ell v^k\|_{\cB_{p_{k,\ell}}} \|v\|_{H^{\frac12}}^{\frac2{p_{k,\ell}}}
\|\nabla v\|_{H^{\frac12}}^{2-\frac2{p_{k,\ell}}},
\end{split}
\] which along with Inequalities\refeq{bonydecomp},
\refeq{lemtheoblowupendpointdemoeq1}
and\refeq{lemtheoblowupendpointdemoeq2} yields the lemma.
\end{proof}

We now turn to the proof of Theorem\refer{blowupBesovendpoint} and
Theorem\refer{theoxplosaniscaling}.

\begin{proof}[Conclusion of the proof of
Theorem\refer{blowupBesovendpoint}] We shall prove that, for any~$T$
less than~$T^\star$, \beq \label{a.11} \|v(T)\|^2_{H^{\frac12}} +
\int_0^T \|\nabla v(t) \|^2_{ H^{\frac12}} dt \leq \|v_0\|_{H^{\frac
12} }^2\exp \biggl(C \sum_{1\leq k,\ell\leq3} \int_0^{T}
\|\partial_\ell v^k(t)\|^{p_{k,\ell}}_{\cB_{p_{k,\ell}}} dt\biggr).
\eeq As a matter of fact, we get, by taking $H^{\frac12}$ energy
estimate to $(NS)$ and Lemma \ref{lemtheoblowupendpoint}, that
$$
\longformule{
\frac 12\frac d{dt} \|v(t)\|_{\dH^{\frac12}}^2
+\|\nabla
v(t)\|_{\dH^{\frac12}}^2 = -\bigl( v\cdot\na v | v\bigr)_{H^{\frac12}}
}
{
{}\lesssim
  \sum_{ k,\ell=1}^3\|\partial_\ell v^k(t)\|_{\cB_{p_{k,\ell}}} \|v(t)\|_{H^{\frac12}}^{\frac2{p_{k,\ell}}}
\|\nabla v(t)\|_{H^{\frac12}}^{2-\frac2{p_{k,\ell}}}.
}
$$
Using the convexity inequality, we infer that
$$
\frac d{dt} \|v(t)\|_{\dH^{\frac12}}^2 + \|\nabla
v(t)\|_{\dH^{\frac12}}^2 \lesssim \|v(t)\|^2_{H^{\frac12}}\Bigl(
 \sum_{ k,\ell=1}^3\|\partial_\ell v^k(t)\|^{p_{k,\ell}}_{\cB_{p_{k,\ell}}}\Bigr) .
$$
Gronwall lemma implies \eqref{a.11}. This completes the proof of
Theorem\refer{blowupBesovendpoint}.
\end{proof}


\begin{proof}[Conclusion of the proof of
Theorem\refer{theoxplosaniscaling}]  We are going to deduce
Theorem\refer{theoxplosaniscaling} from
Theorem\refer{blowupBesovendpoint} and
Proposition\refer{SophisticatedGronwall}. It follows from Lemma
\ref{lemBern} that \beno \max_{1\leq\ell\leq 3} \|\p_\ell
v^3\|_{\cB_p}\lesssim
\sup_{j\in\Z}2^{j\bigl(-1+\f2p\bigr)}\|\D_jv^3\|_{L^\infty}\lesssim
\sup_{j\in\Z}2^{j\bigl(\f12+\f2p\bigr)}\|\D_jv^3\|_{L^2}\lesssim
\|v^3\|_{\dH^{\f12+\f2p}}, \eeno which together with \eqref{k.1}
ensures that \beq
 \label{b.29} \max_{1\leq\ell\leq
3}\int_0^{T^\star}\|\p_\ell v^3(t)\|_{\cB_p}^p\,dt\lesssim
\int_0^{T^\star}\|v^3(t)\|_{\dH^{\f12+\f2p}}^p\,dt<\infty.
\eeq The
same argument yields
 \beq
 \label{b.30}
\forall T<T^\star\,,\   \int_0^{T}\|\partial_{\rm h}^2\D_{\rm h}^{-1}\p_3 v^3(t)\|_{\cB_p}^p\,dt\lesssim
\int_0^{T}\|v^3(t)\|_{\dH^{\f12+\f2p}}^p\,dt.
 \eeq
While for any integer $N,$ we get by using Lemma \ref{lemBern} that,
for any function~$a$ and $p>\f32,$
\beno
 \|a\|_{\cB_p} &\leq&
\sum_{j\leq N}2^{j\bigl(-2+\f2p\bigr)}\|\D_ja\|_{L^\infty}
+\sum_{j>N}2^{j\bigl(-2+\f2p\bigr)}\|\D_ja\|_{L^\infty}\\
&\lesssim &
\sum_{j\leq N}2^{\f{2j}p}\|a\|_{L^{\f32}}+\sum_{j>
N}2^{j\bigl(-\f43+\f2p\bigr)}\|\na a\|_{L^{\f95}}\\
&\lesssim&2^{\f{2N}p}\|a\|_{L^{\f32}}+2^{N\bigl(-\f43+\f2p\bigr)}\|\na
a\|_{L^{\f95}}. \eeno Choosing $N=\ds \biggl[\log_2\biggl(e+
\Bigl(\frac {\|\na a\|_{L^{\f95}}}
{\|a\|_{L^{\f32}}}\Bigr)^{\f34}\biggr)\biggr],$ we obtain
$$
  \|a\|_{\cB_p}\lesssim
\|a\|_{L^{\f32}}^{1-\f3{2p}}\|\na a\|_{L^{\f95}}^{\f3{2p}}.
$$
Applying this inequality with~$a=\partial_{\rm h}^2\D_{\rm h}^{-1}\om$, we get
$$
 \|\partial_{\rm h}^2\D_{\rm h}^{-1}\om\|_{\cB_p}\lesssim
\|\partial_{\rm h}^2\D_{\rm h}^{-1}\om\|_{L^{\f32}}^{1-\f3{2p}}\|
\partial_{\rm h}^2\D_{\rm h}^{-1}\na\om\|_{L^{\f95}}^{\f3{2p}}.
$$
Once noticed that~$L^p= L^p_{\rm v}(L^p_{\rm h})$, we apply Riesz theorem in the horizontal variables to infer that
$$
\|\partial_{\rm h}^2\D_{\rm h}^{-1}\om\|_{L^{\f32}}\lesssim \|\om\|_{L^{\f32}}\andf
\| \partial_{\rm h}^2\D_{\rm h}^{-1}\na\om\|_{L^{\f95}} \lesssim \|\na
\om\|_{L^{\f95}}.
$$
Then due to\refeq{BiotSavartomegademoeq2}, we deduce that
 \beno
  \|\partial_{\rm h}^2\D_{\rm h}^{-1}\om\|_{\cB_p}\lesssim
\|\om\|_{L^{\f32}}^{1-\f3{2p}}\|\na
\om\|_{L^{\f95}}^{\f3{2p}}\lesssim
\|\om\|_{L^{\f32}}^{1-\f3{2p}}\bigl\|\na\om_{\frac34}\bigr\|_{L^2}^{\f2p},
\eeno Together with 
\eqref{b.30}, this gives, for any~$T$ less than~$T^\star$,
$$
{
 \int_0^{T}\|\nabla_{\rm h}v^{\rm
h}(t)\|_{\cB_p}^p\,dt\lesssim
\int_0^{T}\|v^3(t)\|_{\dH^{\f12+\f2p}}^p\,dt
}
{
{}+ \sup_{t\in [0,T[}\|\om(t)\|_{L^{\f32}}^{p-\f32}
\int_0^{T}\bigl\|\na\om_{\f34}(t)\bigr\|_{L^2}^2\,dt.
}
$$
Proposition\refer{SophisticatedGronwall} then implies that \beq
\label{b.31}
\begin{split} \int_0^{T^\star}\|\nabla_{\rm h}v^{\rm
h}(t)\|_{\cB_p}^p\,dt<\infty.
\end{split}
\eeq

Let us  observe  from \eqref{a.1} and \eqref{a.1wrt} that the
components of $\partial_3v^{\rm h}$ are sum of terms of  the
form~$\partial_{\rm h}\D_{\rm h}^{-1}
\partial_3 f$ with~$f=\om$ or~$\partial_3v^3$.  On the one hand, we get, by applying Lemma
\ref{lemBern}, that \beno \|\D_j\p_3v^{\rm h}_{\rm
div}(t)\|_{L^\infty} & \lesssim &\sum_{\substack{k\leq
j+N_0\\\ell\leq j+N_0}}\|\D_k^{\rm h}\D_\ell^{\rm v}\nabla_{\rm
h}\D_{\rm h}^{-1}\p_3^2
v^3(t)\|_{L^\infty}\\
&\lesssim &\sum_{\substack{k\leq j+N_0\\\ell\leq
j+N_0}}2^{\f{\ell}2}\|\D_k^{\rm h}\D_\ell^{\rm v}\p_3^2v^3(t)\|_{L^2}\\
 & \lesssim & \|\p_3^2v^3(t)\|_{\cH_\th}\sum_{\substack{k\leq j+N_0\\\ell\leq j+N_0}}
2^{k(\f12-\th)}2^{\ell(\f12+\th)}\\
 &  \lesssim &  2^j\|\p_3^2v^3(t)\|_{\cH_\th}.
\eeno
Together with Definition \ref{def2.2}  this implies
 \beno
\|\p_3v^{\rm h}_{\rm div}(t)\|_{\cB_2}\lesssim
\|\p_3^2v^3(t)\|_{\cH_\th}. \eeno
Proposition\refer{SophisticatedGronwall} implies  that
 \beq
\label{b.32}
\int_0^{T^\star}\|\p_3v^{\rm h}_{\rm div}(t)\|_{\cB_2}^2\,dt\lesssim
\int_0^{T^\star}\|\p_3^2v^3(t)\|_{\cH_\th}^2\,dt<\infty.
\eeq

On the other hand, we deduce from Lemma \ref{lemBern} that
\beno
\|\D_j\p_3v^{\rm h}_{\rm curl}(t)\|_{L^\infty}
&\lesssim &
2^{\f{2j}3}\sum_{k\leq j+N_0}2^{\f{k}3}\|\D_j\D_k^{\rm h}\p_3\om(t)\|_{L^{\f32}}\\
&\lesssim & 2^j\|\p_3\om(t)\|_{L^{\f32}}, \eeno from which and
\eqref{estimbasomega34}, we infer that for any~$T$ less
than~$T^\star$, \beno \int_0^{T}\|\p_3v^{\rm h}_{\rm
curl}(t)\|_{\cB_2}^2\,dt
& \lesssim & \int_0^{T}\|\p_3\om(t)\|_{L^{\f32}}^2\,dt\\
 & \lesssim & \sup_{t\in
[0,T]}\bigl\|\om_{\f34}(t)|\|_{L^2}^{\f23}\int_0^{T}\bigl\|\na
\om_{\f34}(t)\bigr\|^2_{L^2}\,dt. \eeno
Proposition\refer{SophisticatedGronwall} then implies that
$$
\int_0^{T^\star}\|\p_3v_{\rm curl}^{\rm h}(t)\|_{\cB_2}^2\,dt<\infty.
$$
With Inequalities\refeq{b.29}, \eqref{b.31},\refeq{b.32}, and by
virtue of Theorem \ref{blowupBesovendpoint}, we   conclude the proof
of Theorem\refer{theoxplosaniscaling}.
\end{proof}

\noindent {\bf Acknowledgments.} Part of this work was done when
J.-Y. Chemin was visiting Morningside Center of the Academy of
Mathematics and Systems Sciences, CAS. We appreciate the hospitality
and the financial support from MCM and AMSS. P. Zhang is partially
supported by NSF of China under Grant 10421101 and 10931007, and
innovation grant from National Center for Mathematics and
Interdisciplinary Sciences.

\end{document}